\documentclass[12pt]{article}

\usepackage{amsmath,amssymb,amsthm}
\usepackage{verbatim}
\usepackage{color}
\usepackage{graphicx}

\newtheorem{theorem}{Theorem}[section]
\newtheorem{thm}[theorem]{Theorem}
\newtheorem{corollary}[theorem]{Corollary}
\newtheorem{cor}[theorem]{Corollary}
\newtheorem{lemma}[theorem]{Lemma}
\newtheorem{lem}[theorem]{Lemma}
\newtheorem{proposition}[theorem]{Proposition}
\newtheorem{prop}[theorem]{Proposition}

\newtheorem{remark}[theorem]{Remark}

\numberwithin{equation}{section}

\def\be{\begin{equation}}
\def\ee{\end{equation}}
\def\bes{\begin{equation*}}
\def\ees{\end{equation*}}

\newcommand{\eqn}[2]{\begin{equation}\label{#1}#2\end{equation}}
\newcommand{\eqnst}[1]{\begin{equation*}#1\end{equation*}}
\newcommand{\eqnspl}[2]{\begin{equation}\begin{split}\label{#1}%
	#2\end{split}\end{equation}}
\newcommand{\eqnsplst}[1]{\begin{equation*}\begin{split}%
	#1\end{split}\end{equation*}}

\def\es{\emptyset}

\def\Z{\mathbb{Z}}
\def\Capac{\mathrm{Cap}}

\def\cR{\mathcal{R}}

\def\cT{\mathcal{T}}

\def\eqd{{\buildrel (d) \over =}}

\definecolor{darker-green}{rgb}{0,0.75,0}

 \def\sB {{\mathcal B}} 
 \def\sE {{\mathcal E}} \def\sF {{\mathcal F}}
  
 \def\sK {{\mathcal K}} \def\sL {{\mathcal L}}
  
\def\sP {{\mathcal P}}  \def\sR {{\mathcal R}}
  \def\sU {{\mathcal U}}

 \def\bE {{\mathbb E}} 
 \def\bH {{\mathbb H}}

\def\bP {{\mathbb P}}  \def\bR {{\mathbb R}}

 \def\bZ {{\mathbb Z}}

\def\med{\medbreak\noindent}

\def\sms{\smallskip}
\def\ms{\medskip}

\def\sm{\smallskip\noindent}

\def\ignore#1{}

\def\al {\alpha}
\def\lam {\lambda}  
 
 \def\Th{\Theta}

 \def\gam{\gamma}

\def\to {\rightarrow}

\def\pd {\partial}
\def\q{\quad} 
\def\dint{\int\kern-.6em\int}

\def\Cap{\mathop{{\rm Cap }}}

\def \half {{\tfrac12}}
\def\fract{\tfrac}

\def\wt{\widetilde}

\def\be{\begin{equation}}
\def\ee{\end{equation}}
\def\bes{\begin{equation*}}
\def\ees{\end{equation*}}
\def\ba{\begin{align}}
\def\ea{\end{align}}
\def\xxea{\end{align}}
\def\bas{\begin{align*}}
\def\eas{\end{align*}}

\def\nn{\nonumber} 

\def\proof{{ \sm {\em Proof. }}}
\def\qed{{\hfill $\square$ \bigskip}}

\definecolor{dgreen}{rgb}{0, 0.6, 0.1}
\definecolor{dblue}{rgb}{0, 0.0, 0.6}
\definecolor{vdblue}{rgb}{0,.08, 0.45}
\definecolor{dred}{rgb}{0.7, 0.0, 0.0}
\definecolor{vdblue}{rgb}{0,.08, 0.45}

\definecolor{purple}{rgb}{0.6, 0.0, 0.6}
\definecolor{mytext}{rgb}{0.1, 0.1, 0.1}

\def\dU{d_\sU}

\begin{document}

\font\titlefont=cmbx14 scaled\magstep1
\title{\titlefont Geometry of the uniform spanning forest components in high dimensions}

\author{
M. T. Barlow\footnote{Research partially supported by NSERC (Canada)},
A. A. J\'arai
}

\maketitle

\begin{abstract}
In this note we study the geometry of the component of the origin 
in the Uniform Spanning Forest of $\mathbb{Z}^d$, as well as
in the Uniform Spanning Tree of wired subgraphs of $\bZ^d$, 
when $d \ge 5$. In particular, we study connectivity properties with 
respect to the Euclidean and the intrinsic distance.
We intend to supplement these with further estimates 
in the future. We are making this preliminary note 
available, as one of our estimates is used in work of
Bhupatiraju, Hanson and J\'arai \cite{BHJ15} on sandpiles.

\end{abstract}

\section{ Introduction } \label{sec:intro}

The Uniform Spanning Tree (UST) on a finite graph $G$ is a random spanning
tree of $G$, chosen uniformly among all spanning trees of $G$.
Motivated by questions of Lyons, Pemantle \cite{Pem91} considered
the weak limit of the USTs on a growing sequence of subgraphs of $\bZ^d$, 
induced by sets $V_n \uparrow \bZ^d$, and showed that the limit 
exists. The limiting random object, that is a random spanning forest 
of $\bZ^d$, is called the Uniform Spanning Forest (USF). 
Implicit in Pemantle's work is the result that an alternative 
choice of boundary condition yields the same limit. Namely, form the 
``wired'' graph $G_n^W = (V_n \cup \{ r_n \}, E_n)$, by collapsing all
vertices in $\bZ^d \setminus V_n$ into $r_n$, and removing self-loops created 
at $r_n$. Then the weak limit of the USTs on $G_n^W$ coincides with the USF.
One of Pemantle's results was that the USF is connected a.s.~in dimensions $1 \le d \le 4$, 
but it consists of infinitely many (infinite) trees a.s.~in dimensions $d \ge 5$. 

Fundamental to the study of the UST/USF is Wilson's algorithm \cite{W}, \cite{LP:book} 
that allows one to build the UST/USF from Loop-Erased Random Walks (LERWs), 
and thereby analyze it in terms of random walk. All the necessary background 
about the UST/USF, that we do not detail in this note, can be found in 
the book \cite{LP:book}.

Masson \cite{Mas} and Barlow and Masson \cite{BM1,BM2} studied 
the geometry of the LERW and the UST in two dimensions. This led 
to a detailed understanding of random walk on the UST. 
The purpose of this note is to prove estimates on the geometry 
of the LERW and the USF in dimensions $d \ge 5$. We are interested 
in properties such as the length of paths and volumes of balls,
both with respect to Euclidean distance and the intrinsic 
metric of the tree components. On the one hand we are interested 
in extending results from 2D to high dimensions, where the geometry
is very different. On the other hand, our Theorem \ref{T:cyclepop} 
is used in work of Bhupatiraju, Hanson and J\'arai \cite{BHJ15} 
on sandpiles.

\section{Notation}
\label{sec:notation}

Let $\sU=\sU_{\bZ^d}$ be the USF in $\bZ^d$, viewed as a random 
subgraph of the nearest neighbour integer lattice. 
Write $\sU(x)$ for the connected component of $\sU$ containing $x$. 

We extend this notation to $D \subset \bZ^d$ as follows.
When $D$ is finite, $\sU = \sU_D$ denotes the UST on the wired graph 
$G_D^W = (D \cup \{ r_D \}, E_D)$, where $E_D$ can be 
identified with those edges of $\bZ^d$ that have at least one
endpoint in $D$. For $x \in D$, we denote by $\sU(x)$ the 
connected component of $x$ in the graph obtained from $\sU$
by splitting all edges at $r_D$. In other words, $\sU(x)$ is the
union of those paths in $\sU$ that do not contain $r_D$ as an
interior vertex. We write $\sU_0$ for $\sU(0)$, when $0 \in D$.

When $D \subset \bZ^d$ is infinite, we let $\sU$ denote the weak 
limit, as $n \to \infty$, of the USTs on the wired graphs 
$G_{D_n}^W$, where $D_n = \{ x \in D : |x| \le n \}$. 
The limit exists due to monotonicity; see \cite{LP:book}. 
Wilson's algorithm rooted at infinity \cite{BLPS}, \cite{LP:book} can be
easily adapted to sample $\sU$. We let $\sU(x)$ denote the union of
those paths in $\sU$ that do not contain $r_D$ as an interior vertex,
and $\sU_0 = \sU(0)$.

For any of the cases of $\bZ^d$, or $D \subset \bZ^d$ finite or infinite, 
we let
\eqnst
{ d_\sU(x,y)
  := \text{graph distance between $x$ and $y$ in $\sU$}, }
where, if $y \not\in \sU(x)$, we set $d_\sU(x,y) = \infty$.
The meaning of $\sU$ will always be clear from context.

\med 

{\bf Notation for sets:}
We denote balls in different metrics as follows:
\begin{align*}
  B_E(x,r) &=\{ y \in \bZ^d : |x-y| \le r \}, \\
  B_n &= B_E(0,n)\\
  Q(x,n) &=\{ y\in \bZ^d: ||x-y||_\infty \le n \}, \\
  Q_n &= Q(0,n), \\
  B_\sU(x,r) &=\{ y \in \bZ^d : \dU(x,y) \le r \}, 
\end{align*}
For $A \subset \bZ^d$ we denote:
\begin{align*}
 \pd A &= \{ x \in \bZ^d - A: x \sim y \hbox{ for some } y \in A \}, \\
 \pd_i A &= \{ x \in A: x \sim y \hbox{ for some } y \in A^c \}, \\
  A^o &= A - \pd_i A.
\end{align*}
Let $\pi_i$ be projection onto the $i$th coordinate axis, and 
$\bH_n$ be the hyperplane 
$$ \bH_n=\{ x : \pi_1(x) = n \}. $$
Let $\sR_n = \{ n \} \times [-n,n]^{d-1}$ denote the 
``right-hand face'' of $[-n,n]^d$, in the 
first coordinate direction.

\med 

{\bf Notation for processes.} $S^x=(S^x_k, k \ge 0)$ is simple random walk with $S^x_0 =x$, and
$\bP^x$ is its law. We let $S=S^0$,  and $\bP = \bP^0$.  
If we discuss random walks $S^x$ and $S^y$ with $x \not= y$,
then they will always be independent.

A path $\gam$ is a (non-necessarily self avoiding) sequence of adjacent
vertices in $\bZ^d$ -- ie $\gam=(\gam_0, \gam_1, \dots )$ with $\gam_{i-1}\sim \gam_i$.
(Sometimes we will write $\gam(i)$ for $\gam_i$.)
Paths can be either finite or infinite.
We will often need to consider the beginning or final portions
of paths with respect to the first or last hit on a set.
To this end, we define a number of operations on paths. 
Let $\gam= (\gam_0, \gam_1, \dots )$ be a path.
Given a set $A \subset \bZ^d$ define 
$k_1 = \min\{ k \ge 0: \gam_k \in A \}$, 
$k_2 = \max \{ k \ge 0: \gam_k \in A \}$, 
and set
\begin{align*}
 \sB^F_A \gam &= ( \gam_{k_1}, \gam_{k_1 +1}, \dots, ), \\
 \sB^L_A \gam &= ( \gam_{k_2}, \gam_{k_1 +1}, \dots, ), \\
 \sE^F_A \gam &= (\gam_0, \dots,  \gam_{k_1}),\\
 \sE^L_A \gam &= (\gam_0, \dots,  \gam_{k_2}), \\
 \Th_k \gam &= (\gam_k, \dots ), \\
 \Phi_k \gam &= (\gam_0, \dots, \gam_k), \\
  H_A(\gam) &= \sum_i 1_{( \gam_i \in A )}.
\end{align*}
Thus $\sB^F_A \gam$ is the path $\gam$ `\underline{B}eginning' at the `\underline{F}irst' hit
on $A$, and $\sE^L_A \gam$ is the path $\gam$ `\underline{E}nded' at the `\underline{L}ast' hit 
on $A$, etc. If $\gam$ is a finite path we write $|\gam|$ for the
length of $\gam$. $H_A(\gam)$ is the number of hits by $\gam$ on the set $A$.
Let $\sL \gam$ be chronological loop erasure of $\gam$, and
if $\gam=(\gam_0, \dots, \gam_n)$ is a finite path let
$\sR \gam = (\gam_n , \gam_{n-1}, \dots , \gam_0)$ be the time reversal of
$\gam$.

We define hitting times 
\eqnsplst
{ \tau_A  &= \inf \{ j \ge 0 : S_j \not\in A \}, \\
  T_A &= \inf \{ j \ge 0 : S_j \in A \}, \\
   T_A^+ &= \inf \{ j \ge 1 : S_j \in A \}. }
When we need to specify the process we write $T_A[S]$ etc.

Given a domain $D \subset \bZ^d$, we denote the Green functions
\eqnsplst
{ G_D(x,y)
  &= \bE^x \Big( \sum_{0 \le k < \tau_D} I[ S^x = y ] \Big) \\
  G(x,y)
  &= G_{\bZ^d}(x,y). }

\med
{\bf A note on constants.} Throughout, $c$ and $C$ will denote positive
finite constants that only depend on the dimension $d$, and whose value may change
from line to line, and even within a single string of inequalities.

\section{Properties of the LERW}   
\label{sec:lew} 

In this section we derive a number of auxiliary estimates on LERW
in dimensions $d \ge 5$. Some of these will be used in 
Sections \ref{sec:ub} and \ref{sec:vol-lb}, where we give upper and
lower bounds on the volume of balls in the intrinsic metric. 
Two results of this section that are of interest 
in themselves are: (i) Proposition \ref{P:len-lb}, that gives a 
large deviation upper bound on the lower tail of the number of 
steps in a LERW up to its exit from a large box; and (ii) 
Theorem \ref{T:Ass2}, that gives an upper bound on the probability that
$x, y \in \bZ^d$ are in the same component of $\sU$ and the 
path between them has length at most $n$.

The papers \cite{Mas, BM1} give a number of properties  of LERW
in $\bZ^2$, some of which hold for more general graphs.

A fundamental fact about LERWs is the following ``Domain Markov property''
--- see \cite{La2}.

\begin{lemma} \label{L:dmp}
Let $D \subset \bZ^d$,
let $\gam=(\gam_0, \dots , \gam_n)$ be a path from $x=\gam_0$ to $D^c$. 
Set $\al = \Phi_k \gam$, $\beta = \Th_k \gam$.
Let $Y$ be a random walk started at $\gam_k$ conditioned 
on the event $\{ \tau_D(Y) < T^+_\al(Y) \}$. Then
\be
 \bP\big( \sL ( \sE^F_{D^c} S ) = \gam | \Psi_k (  \sL ( \sE^F_{D^c} S ))=\al \big)
 = \bP( \sL(  \sE^F_{D^c} Y ) = \beta ). 
\ee 
\end{lemma}

A key result in \cite{Mas} is a  `separation lemma' when $d=2$ -- see 
\cite[Theorem 4.7]{Mas}.
Let $S, S'$ be independent SRW in $\bZ^d$ with $S_0=S'_0=0$, 
and $T_n, T'_n$ be the hitting times of $\pd Q_n$. Set
\begin{align*}
 F_n &= \{  S[1, T_n] \cap S'[1, T'_n] = \emptyset \}, \\
 Z_n &= d( S(T_n),  S'[1, T'_n]) \vee  d( S'(T'_n),  S[0, T_n]).
\end{align*}

\begin{lem}\label{L:sep} (`Separation lemma').
Let $d\ge 5$. There exists $c_1>0$ such that
$$ \bP( Z_n \ge \half n | F_n) \ge c_1. $$
\end{lem}

\proof  
Let $e_1=(1,0, \dots,0)$. Let $X$ be a SRW started at $2k e_1$, and $A_k=\{ je_1, k\le j \le 2k\}$.
Since $d \ge 5$ two independent SRWs intersect with probability less than 1, and thus
there exists $k$ (depending on $d$) such that
$$ \bP^0( \hbox{$S$ hits } X \cup A_k ) \le \fract{1}{16} d^{-2}. $$
Now fix this $k$, and let 
$$ G_1 =\{  S_i = -i e_1, S'_i = ie_1, 0\le i \le k\}. $$
So $\bP(G_1) = (2d)^{-2k}$.
Then writing $G_2= \{  S[1, T_{n/2}] \cap S'[1, T'_{n/2}]  \neq \emptyset \}$, 
\bas
\bP(G_2|G_1) &\le
\bP( S[k+1, T_{n/2}] \cap S'[1, T'_{n/2}] \neq \emptyset | G_1) 
+ \bP( S[1, T_{n/2}] \cap S'[k, T'_{n/2}] \neq \emptyset | G_1) \\
&\le  \fract18 d^{-2}. 
\end{align*}

Let $H_\pm$ be the left and right faces (in the $e_1$ direction) of the cube
$Q_{n/2}$. We have
$$ \bP( S_{T_{n/2}} \in H_- |G_1 ) \ge (2d)^{-1}. $$
So if $G_3=  G_2^c \cap \{  S_{T_{n/2}} \in H_-, S'_{T'_{n/2}} \in H_+\}$,
\begin{align*}
  \bP( G_3 | G_1) &\ge
  \bP( S_{T_{n/2}} \in H_-,  S'_{T'_{n/2}} \in H_+  |G_1 ) - 
 \bP( G_2   | G_1) \\
&\ge (2d)^{-2} - (8d^2)^{-1} = (8d^2)^{-1}.
\end{align*}
If $G_3$ occurs then let $G_4$ be the event that $S'$ then (i.e. after
time $T'_{n/2}$ leaves $Q_n$ before it hits hits $\bH_0$, 
and $S$ leaves $Q_n$ before it hits
$\bH_0$. By comparison with a one-dimensional SRW each of these events
has probability at least $1/3$, so $\bP(G_4|G_3) \ge 1/9$. 
On the event $G_1 \cap G_3 \cap G_4$ the path $S[0, T_n]$ is contained in 
$[-n,0]\times [-n,n]^{d-1} \cup Q_{n/2}$,  and $\pi_1(S'_{T'_n}=n)$, so that
$d(S'_{T'_n} ,S[0, T_n]) \ge n/2$. The same bound holds if we interchange
$S'$ and $S$, and so we deduce that
$$ \bP(Z_n \ge \half n|F_n ) \ge  \bP(\{Z_n \ge \half n\} \cap F_n)
\ge \bP( G_1 \cap G_3 \cap G_4) \ge (2d)^{-2k} (8d^2)^{-1} 9^{-1}. $$
\qed

\sm {\bf Remark.}
The result in $d \ge 5$ is much easier than $d =2$, since with high probability
$S$ and $S'$ do not interesect. The 
proof for $d=2$ uses the fact that if the two processes get too close,
then by the Beurling estimate they hit with high probability. 

\ms
In the remainder of this section we give some estimates on the length of 
LERW paths in $\bZ^d$ with $d \ge 5$. We fix $D \subset \bZ^d$ and
$N \ge 1$ such that $Q_N = Q(0,N) \subset D$. We will be interested in 
the number of steps the LERW from $0$ to $\pd D$ takes up to its
first exit from $Q_N$. Let $S$ be SRW on $\bZ^d$ with $S_0=0$. 
Let $$ L= \sL ( \sE^F_{D^c} (S)). $$
In words, $L$ is the loop erasure of $S$ up to its first hit
on the boundary of $D$.

Our estimate will be broken down into studying $L$ in `shells'
$Q_{n+m} \setminus Q_n$. For this purpose, let us fix $n,m$ such that 
$16 \le n < n+m \le N$, with $m \le n/8$.
Let $$ \al = \sE^F_{ \pd_i Q_n } L , \q L' = \sB^F_{\pd_i Q_n}L. $$
So $\al$ is the path $L$ up to its first hit on $\pd_i Q(0,n)$, and
$L'$ is the path of $L$ from this time on. 
See Figure \ref{fig:box-setup}.

Let us condition on $\al$. Let $x_0 \in \pd_i Q_n$ be the endpoint 
of $\al$. When $x_0 \in \bH_n$, we let $ x_1 = x_0 + (m/2)e_1 $
and set
$$ A=A(x_0) = Q(x_1, m/4 ), 
\quad, A^* = Q(x_1, 3m/8 ). $$
When $x_0$ lies on one of the other
faces of $Q_n$, we replace $e_1$ by the unit vector pointing
towards that faces to define $x_1$ and $A(x_0)$. 
See Figure \ref{fig:box-setup}.
\begin{figure}
\centerline{\includegraphics{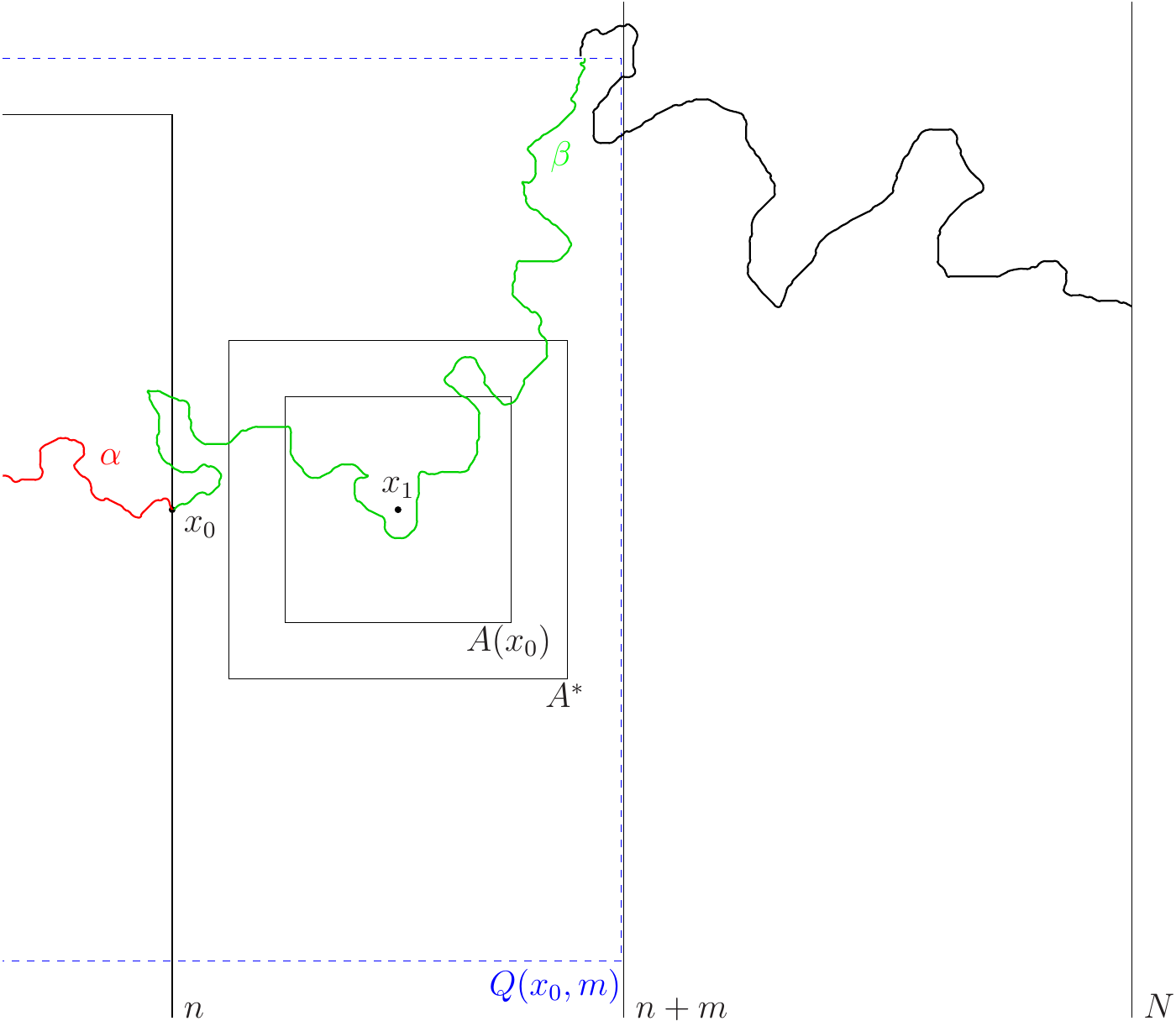}}
\caption{\label{fig:box-setup} Setup and notation for the piece
of the LERW in the shell $Q_{n+m} \setminus Q_n$.}
\end{figure}
Set $$ \beta = \sE^F_{\pd_i Q(x_0,m)} L' . $$

Let $\wt X^z$ be $S^z$ conditioned on $\{ \tau_D < T^+_\al\}$. 
While the process $\wt X^z$ depends on $\al$, our notation will not
emphasize this point.
Write $\wt X$ for $\wt X^{x_0}$, and $\wt G_D(x,y)$ for the Green
function for $\wt X^x$. By the domain Markov property, Lemma \ref{L:dmp}, 
we have (conditional on $\al$) that 
\be
  L'  \, \, \eqd  \, \, 
  \sL( \sE^F_{\pd D} \wt X ). 
\ee
We write $\wt T$, $\wt \tau$, etc.~for hitting and exit times by $\wt X$.
Set
$$ h(x) = \bP^x( \tau_D < T_\al). $$
Then 
\be \label{e:wtG}
 \wt G_D(x,y) = \frac{ h(y)}{h(x)} G_D(x,y), \, x,y \in D- \alpha.
\ee
The standard Harnack inequality 
(see \cite{La2}) gives
\be \label{e:hiA}
 h(y) \asymp h(x_1), \q y \in A^*, 
\ee
and thus
\be \label{e:hGn}
   \wt G_D(x,y) \asymp G_D(x,y), \, x,y \in A^*.
\ee   

\begin{lemma} \label{L:Mub}
Let $d \ge 3$. For any $\al$ we have
\begin{align}
 \bE( H_A(\beta)| \al)  &\le c_1 m^2, \\
 \bE( H_A(\beta)^2|\al)  &\le c_1 m^4.
\end{align}
\end{lemma}

\proof
This is a standard computation with Green functions. Let
$B = Q(x_0, m)$. Then, since $\beta$ is a subset of the 
path of $\wt X$, we have 
$$ H_A(\beta) \le  \sum_{k=0}^{\wt \tau_{B}} 1_{ (\wt X \in A) }
  = H_A( \sE^F_{\pd_i B} \wt X) =: \wt H. $$
Then for $p=1,2$, 
$$ \bE^{x_0} (\wt H^p| \al)
= \bE^{x_0}\Big( 1_{( \wt T_{A^*} <\wt \tau_B)} \bE^{\wt X_{\wt T_{A^*}}}( {\wt H}^p) \Big)
\le \max_{ z\in \pd_i A^*} \bE^z \wt H^p.  $$
Let $z \in \pd_i A^*$. Then using \eqref{e:hGn}
\begin{align*}
 \bE^z (\wt H|\al)  &= \sum_{y \in A} \wt G_B(z, y)
 \le c |A| \max_{y \in A} G_B(z,y) \le c' m^2 m^{2-d} = c' m^2. 
\end{align*}
Also since on $A^*$ we have
$\wt G_B \asymp G_B \le G$,
\begin{align*}
 \bE^z (\wt H^2|\al)  &\le 2 \sum_{k=0}^\infty \sum_{k=j}^\infty
1_{( k\le \wt \tau_B)} 1_{( j \le \wt \tau_B)} 1_{(\wt X_k \in A)}
1_{(\wt X_j \in A)} \\
&\le 2 \sum_{x \in A} \sum_{ y \in A} \wt G_B(z,x)\wt G_B(x,y) \\
&\le c |A| m^{2-d} \max_{x \in A} \sum_{ y \in A} \wt G(x,y)
 \le c' m^4.  
\end{align*}
\qed 

\begin{remark}
The same argument works if we consider
$\bE( H_{Q(x_1, \lam m)}(\beta)^p |\al)$, $p=1,2$, for any $\lam \in (0, \frac12)$.
\end{remark}

We now turn to the harder problem of obtaining a lower bound on $\bE H_A(\beta)$,
and begin with a boundary Harnack inequality which extends 
\cite[Proposition 3.5]{Mas} to higher dimensions. 
See \cite{BK} for further extensions.
In what follows $\sR_m = \bH_m \cap Q_m$ is the `right hand face' of $Q_m$.

\begin{lemma} \label{L:ubh}
Assume $d \ge 1$. Let $\sK$ be an arbitrary nonempty subset of 
$[-m+1,0] \times [-m+1,m-1]^{d-1}$. For all $m \ge 1$ 
and all $\sK$ we have
\be \bP^0 \big( S(\tau_{Q(0,m-1)}) \in \sR_m \,\big|\, \tau_{Q(0,m-1)} < T^+_\sK \big) \ge (2d)^{-1}.
\ee
\end{lemma}

\begin{proof}
Let $h(z) = \bP^z \big( S_{\tau_{Q(0,m-1)}} \in \sR_m \big)$, $z \in Q(0,m-1)$. 
By symmetry we have $h(0)= 1/2d$. We first show that 
\eqn{e:max-property}
{ h(z)  
  \le h(0) \text{ for all 
      $z \in ([-m+1,0] \times [-m+1,m-1]^{d-1}) \cap \bZ^d$}. }
      
Let $z' = (0,z_2,\dots,z_d)$. 
Let $X^z$ and $X^{z'}$ be  simple random walks with
starting points $z$ and $z'$ respectively; we have
$ h(z) = \bP( X^z_{\tau_A} \in \sR_n)$, with a similar expression
for $h(z')$.
We couple these random walks 
by taking $X^z = z+ S$, $X^{z'}=z'+S$, where $S$ is a 
SRW with $S_0=0$. Then
$\{ X^{z}_{\tau_A} \in \sR_n\} \subset \{ X^{z'}_{\tau_A} \in \sR_n\}$,
and so $h(z) \le h(z')$.

To prove that $h(z') \le h(0)$ we use a coupling of continuous
time random walks $Y$, $Y'$ with $Y_0=0$, $Y'_0=z'$; these have the
same exit distribution as the discrete time walk $S$.
Recall that $\pi_j$ is the projection onto the $j$th coordinate
axis, so that $\pi_j(Y_t)$ gives the $j$th coordinate of $Y_t$;
each coordinate is a continuous time simple random walk (run at rate $1/d$) on
$\bZ$.

The coupling is as follows. If at time $t$ we have
$\pi_j (Y_t)=\pi_j(Y'_t)$ then we run the two $j$th coordinate
processes together, so $\pi_j (Y_{t+s})=\pi_j(Y'_{t+s})$ for
all $s \ge 0$

Note that  we have $|\pi_j (Y_t)| \le |\pi_j(Y'_t)|$ when $t=0$;
the coupling will preserve this inequality for all $t \ge 0$. 
If $|\pi_j (Y_t) - \pi_j(Y'_t)|\ge 2$ then we use reflection coupling,
so that $\pi_j (Y_t)$ and $\pi_j(Y'_t)$ jump at the same time,
and in opposite directions. 
Finally, suppose that $|\pi_j (Y_t) - \pi_j(Y'_t)| =1$,
and let $a=\pi_j (Y_t)$, $a+1= \pi_j (Y'_t)$.
We take three independent Poisson processes on $\bR_+$,
$\sP_1, \sP_2, \sP_3$; each with rate $1/2d$, and make the first
jump of either $\pi_j (Y)$ or $\pi_j(Y')$ after time $t$ to be
at time $t+T$, where
$T$ is the first point in $\sP_1 \cup \sP_2 \cup \sP_3$.
If $T \in \sP_1$  we set
$\pi_j (Y_{t+T}) = a-1$, $\pi_j (Y'_{t+T}) = a+2$.
If $T \in \sP_2$ then we set
$\pi_j (Y_{t+T}) = a+1$, $\pi_j (Y'_{t+T}) = a+1$, and if $T \in \sP_3$ then 
$\pi_j (Y_{t+T}) = a$, $\pi_j (Y'_{t+T}) = a$.
With this coupling we have
$\{ Y'_{\tau_A(Y') } \in \sR_n\} \subset \{ Y_{\tau_A(Y)} \in \sR_n\}$,
and so $h(z') \le h(0)$.

Stopping the bounded martingale $h(S(k))$ at $\tau_{Q(0,m-1)} \wedge T_\sK$, and 
using \eqref{e:max-property} we get
\eqnsplst
{ h(0)
  &= \sum_{y \in \sK} h(y) \bP^0 \big( S(\tau_{Q(0,m-1)} \wedge T^+_\sK) = y \big)
    + \bP^0 \big( \tau_{Q(0,m-1)} < T^+_\sK,\, S(\tau_{Q(0,m-1)}) \in \sR_m \big) \\
  &\le h(0) \bP \big( \tau_{Q(0,m-1)} > T^+_\sK )
    + \bP( \tau_{Q(0,m-1)} < T^+_\sK, \, S(\tau_{Q(0,m-1)}) \in \sR_m \big). }
Rearranging gives the statement of the lemma. \qed
\end{proof}

\sms
We will also need two extensions of Lemma \ref{L:ubh} that we
prove next.

\begin{lemma} \label{L:ubh-ext}
Assume $d \ge 3$. Let $N \ge 1$ and $Q_{4N} \subset D \subset \bZ^d$.
Let $8 \le m \le N/2$ and $n \le N$.
Suppose that $\sK$ is an arbitrary nonempty subset of $Q_n$, and
$x_0 \in \sK \cap \bH_n$. Let $z_0 = x_0 + m e_1$.
There exists a constant $c = c(d) > 0$ such that 
\be \bP^{z_0}( T_{Q(x_0,m/2)} > \tau_D \,|\, T_{\sK} > \tau_D ) 
    \ge c.
\ee
\end{lemma}

\begin{proof}
It is easy to see that the statement holds when $m \ge n/8$, since then
$\bP^{z_0} ( T_{Q_{n+m/2}} > \tau_D ) \ge \bP^{z_0} ( T_{Q_{n+m/2}} = \infty ) \ge c$.
Henceforth we assume that $m < n/8$.

Let $f(z) = \bP^z ( T_{\sK} > \tau_D )$ and 
$g(z) = \bP^z ( T_{\sK} \wedge T_{Q(x_0,m/2)} > \tau_D)$, so that we have to 
prove $f(z_0) \le C g(z_0)$. Let $z_1 = x_0 + 8 m e_1$. Due to the
Harnack principle, it is sufficient to show that $f(z_1) \le C g(z_1)$.

We first show that for all $y \in \pd Q(x_0, 8 m)$ we have
$g(y) \le C g(z_1)$. Let us write $\bH$ for
the hyperplane $\bH_{n + 4 m}$, and $\bH'$ for the hyperplane
$\bH_{n + 2 m}$. Observe that $\bH$ and $\bH'$ are both disjoint 
from $\sK \cup Q(x_0,m/2)$, and they both separate 
$\sK \cup Q(x_0,m/2)$ from $z_1$.

If $y \in \pd Q(x_0, 8 m)$ lies on the same side of $\bH'$
as $z_1$, then $y$ is at least distance $m$ from $\sK \cup Q(x_0,m/2)$, 
and this is comparable to the distance between $y$ and $z_1$. 
Hence for such $y$, the Harnack principle
implies $g(y) \le C g(z_1)$. 

Suppose now that $\bH'$ separates $y$ from $z_1$. Let $Q^{(1)}$ and $Q^{(2)}$
be cubes that are both translates of $Q_{2N}$, such that:\\
(i) the right hand face of $Q^{(1)}$ and the left hand face of $Q^{(2)}$ coincide;\\
(ii) the common set $\cR = Q^{(1)} \cap Q^{(2)}$, is contained in $\bH$; \\
(iii) the center of $\cR$ (viewed as a $(d-1)$-dimensional cube), 
is the point $x_0 + 4 m e_1$. \\
Since $g(S(n \wedge \tau_{Q^{(1)}}))$ is a submartingale under $\bP^y$, we have
\eqn{e:submart}
{ g(y) 
  \le \bE^y ( g(\tau_{Q^{(1)}}) )
  = \sum_{w \in \pd Q^{(1)} \setminus \cR} g(w) \, \bP^y ( S(\tau_{Q^{(1)}}) = w )
    + \sum_{u \in \cR} g(u) \, \bP^y ( S(\tau_{Q^{(1)}}) = u ). }
Since $g(S(n \wedge \tau_{Q^{(2)}}))$ is a martingale under $\bP^{z_1}$, 
we also have
\eqn{e:mart}
{ g(z_1) 
  = \bE^{z_1} ( g(\tau_{Q^{(2)}}) )
  = \sum_{w' \in \pd Q^{(2)} \setminus \cR} g(w') \, \bP^{z_1} ( S(\tau_{Q^{(2)}}) = w' )
    + \sum_{u \in \cR} g(u) \, \bP^{z_1} ( S(\tau_{Q^{(2)}}) = u ). }
The mirror symmetry between $Q^{(1)}$ and $Q^{(2)}$, as well as the Harnack
principle implies that
\eqnsplst
{ \bP^y ( S(\tau_{Q^{(1)}}) = u ) 
  &\le C \bP^{z_1} ( S(\tau_{Q^{(2)}}) = u ) \\
  \bP^y ( S(\tau_{Q^{(1)}}) = w ) 
  &\le C \bP^{z_1} ( S(\tau_{Q^{(2)}}) = w' ), }
where $w'$ is the mirror image of $w \in \pd Q^{(1)} \setminus \cR$
in the hyperplane $\bH$. We also have $g(w) \le 1$, $w \in \pd Q^{(1)} \setminus \cR$,
and $g(w') \ge c$, $w' \in \pd Q^{(2)}$. These observations and \eqref{e:submart}
and \eqref{e:mart} together imply $g(y) \le C g(z_1)$. 

We now show the desired inequality $f(z_1) \le C g(z_1)$. 
Let $1 \le R < \infty$ denote the random variable that counts the
number of times $S^{z_1}$ makes a crossing from $\pd Q(x_0, 8 m)$
to $Q(x_0,m/2)$ before $T_\sK \wedge \tau_D$. We have
\eqnst
{ \bP^{z_1} ( R \ge \ell )
  \le \left( \max_{y \in \pd Q(x_0, 8 m)} \bP^{y} ( T_{Q(x_0,m/2)} < \infty ) \right)^\ell
  \le \gamma^\ell }
with some $0 < \gamma = \gamma(d) < 1$.

Using the strong Markov property at the time when the $\ell$-th crossing
has occurred, we can write 
\eqnsplst
{ f(z_1)
  &= \sum_{\ell = 0}^\infty \bP^{z_1} ( R = \ell,\, T_\sK > \tau_D )
  = g(z_1) + \sum_{\ell = 1}^\infty \bP^{z_1} ( R = \ell,\, T_\sK > \tau_D ) \\
  &\le g(z_1) + \sum_{\ell = 1}^\infty \bP^{z_1} ( R \ge \ell ) \,
     \max_{z \in Q(x_0,m/2)} \bP^z ( (T_{Q(x_0,m/2)} \wedge T_\sK) \circ \Theta_{\tau_{Q(x_0, 8 m)}} > \tau_D ) \\
  &\le g(z_1) + \sum_{\ell = 1}^\infty \gamma^\ell \max_{y \in \pd Q(x_0, 8 m)} g(y) \\
  &\le g(z_1) + C g(z_1). }
This completes the proof of the Lemma.
\end{proof}

\begin{lemma} \label{L:ubh-tauD}
Assume $d \ge 3$. Let $N \ge 1$ and $Q_{4N} \subset D \subset \bZ^d$.
Let $8 \le m \le N/2$ and $n \le N$.
Suppose that $\sK$ is an arbitrary nonempty subset of $Q_n$, and
$x_0 \in \sK \cap \bH_n$. Let $\cR_{n,m}$ denote the 
right hand face of $Q(x_0,m)$. There exists a constant $c = c(d) > 0$ 
such that 
\be \bP^{x_0} \big( S(\tau_{Q(x_0,m)}) \in \cR_{n,m} \,\big|\, T^+_\sK > \tau_D \big) 
    \ge c.
\ee
\end{lemma}

\begin{proof}
Let $\sK_0 = \sK \cap Q(x_0, 2 m)$ and 
$\sK_1 = \sK \setminus \sK_0 = \sK \setminus Q(x_0, 2 m)$.
Due to the boundary Harnack inequality, Lemma \ref{L:ubh}, we have
\eqn{e:simple-ubh}
{ \bP^{x_0} \big( S(\tau_{Q(x_0,m)} \in \cR_{n,m} \,\big|\, T^+_\sK > \tau_{Q(x_0,m)} \big)
  \ge (2d)^{-1}. }
Let $Z$ denote the process that is $S$ conditioned on 
$T_{\sK_1} > \tau_D$. Then \eqref{e:simple-ubh} and an application of the
Harnack principle implies that 
\eqn{e:cond-simple-ubh}
{ \bP^{x_0} \big( Z(\tau_{Q(x_0,m)} \in \cR_{n,m} \,\big|\, T^+_\sK[Z] > \tau_{Q(x_0,m)}[Z] \big)
  \ge c. }
This in turn implies that 
\eqnspl{e:cond-ubh}
{ &\bP^{x_0} \big( S(\tau_{Q(x_0,m)}) \in \cR_{n,m},\, T^+_\sK > \tau_{Q(x_0,m)},\, T_{\sK_1} > \tau_D \big) \\
  &\qquad \ge c \bP^{x_0} \big( T^+_\sK > \tau_{Q(x_0,m)},\, T_{\sK_1} > \tau_D \big) \\
  &\qquad \ge c \bP^{x_0} \big( T^+_\sK > \tau_D \big). }
Let $z_0 = x_0 + 4 m e_1$. Using the Harnack principle, the left hand side 
of \eqref{e:cond-ubh} can be bounded from above by
\eqnspl{e:away}
{ &\bP^{x_0} \big( S(\tau_{Q(x_0,m)}) \in \cR_{n,m},\, T^+_{\sK} > \tau_{Q(x_0,m)} \big) \, 
    \max_{z \in \cR_{n,m}} \bP^z \big( T_{\sK_1} > \tau_D \big) \\
  &\qquad \le C\, \bP^{x_0} \big( S(\tau_{Q(x_0,m)}) \in \cR_{n,m},\, T^+_{\sK} > \tau_{Q(x_0,m)} \big) \, 
    \bP^{z_0} \big( T_{\sK_1} > \tau_D \big). }
An application of Lemma \ref{L:ubh-ext} (with $2 m$ playing the role of $m/2$)
shows that 
\eqnst
{ \bP^{z_0} \big( T_{\sK_1} > \tau_D \big) 
  \le C \, \bP^{z_0} \big( T_{\sK_1 \cup Q(x_0, 2 m)} > \tau_D \big)
  \le C \, \bP^{z_0} \big( T_{\sK} > \tau_D \big). }
Substituting this into \eqref{e:away}, and using the Harnack principle again,
we get that the right hand side of \eqref{e:away} is bounded above by
\eqnspl{e:bound-away}
{ &C\, \bP^{x_0} \big( S(\tau_{Q(x_0,m)}) \in \cR_{n,m},\, T^+_{\sK} > \tau_{Q(x_0,m)} \big) \, 
    \bP^{z_0} \big( T_{\sK} > \tau_D \big) \\
  &\qquad \le C \, \bP^{x_0} \big( S(\tau_{Q(x_0,m)}) \in \cR_{n,m},\, T^+_{\sK} > \tau_{Q(x_0,m)} \big) \, 
    \min_{z \in \cR_{n,m}} \bP^z \big( T_{\sK} > \tau_D \big) \\ 
  &\qquad \le C \, \bP^{x_0} \big( S(\tau_{Q(x_0,m)}) \in \cR_{n,m},\, T^+_{\sK} > \tau_D \big). }
The inequalities \eqref{e:cond-ubh}, \eqref{e:away} and \eqref{e:bound-away}
together imply the claim of the Lemma.
\end{proof}

\sms
We now return to the task of giving a lower bound for $\bE ( H_A(\beta) )$.
We will need the following lower bound on $\wt G$.

\begin{lemma} \label{L:wtG}
Assume $d \ge 3$. Let $z \in A$. Then
$$ \wt G_D(x_0, z) \ge c m^{2-d}. $$
\end{lemma}

\proof 
This uses the extension of the boundary Harnack inequality, Lemma \ref{L:ubh-tauD}.
Let $V_z$ be the number of hits on $z$ by $\wt X$ before $\wt \tau_D$. 
Let $\wt T = \wt T_{\pd_i Q(x_0, m/8)}$. Note that $Q(x_0, m/8)$ and
$A^*$ intersect on one of the faces of $Q(x_0, m/8)$. Then
since $\wt T < \wt \tau_D$,
\begin{align*}
\wt G_D(x_0, z) &= \bE^{x_0} V_z 
 = \bE^{x_0}\Big( \bE^{\wt X_{\wt T}} V_z \Big) 
\ge \bE^{x_0}\Big( 1_{ (\wt X_{\wt T} \in A^*)} 
   \min_{y \in \pd_i A^*} \bE^y V_z\Big)\\
&=  \bP^{x_0}( \wt X_{\wt T} \in A^*) 
 \min_{y \in \pd_i A^*} \wt G_D(y,z).
\end{align*}
Using \eqref{e:wtG} and \eqref{e:hiA} we have 
$\wt G_D(y,z) \asymp G_D(y,z) \asymp m^{2-d}$ if $y \in \pd_i A^*$.
Let $T=  T_{\pd_i Q(x_o, m/8)}$ (for $S$). Lemma \ref{L:ubh-tauD}
implies
\begin{align*}
 \bP^{x_0}(\wt X_{\wt T} \in A^*)
 = \bP^{x_0}(S_T \in A^*| T^+_\al > \tau_D ) \ge c,
\end{align*}
and the Lemma follows.
\qed 

\ms
The key estimate is the following.

\begin{lemma} \label{L:ptlb}
Assume $d \ge 5$. Then 
\be
 \bE( H_A(\beta) | \al) \ge c m^2. 
\ee
\end{lemma}

\proof 
It is enough to prove that if  $z \in A$ then
\be \label{e:zhit}
\bP( z \in \beta | \al ) \ge c m^{2-d}. 
\ee
Let $Y$ be $\wt X$ conditioned to hit $z$ before 
$\wt T^+_\al \wedge \wt \tau_D$, and let $\wt X^{z}$ be independent 
of $Y$. Let 
$$ Y'= \sE^L_{z}( \sE^F_{\pd D}( Y)), $$
so $Y'$ is the path of $Y$ up to its last hit on $z$ before its
first exit from $D$. Let also $X'= \Th_1 \sE^F_{\pd D} \wt X^z$. 
(We need to apply $\Th_1$ since the last point of $Y'$ and
the first point of $X'$ are both $z$.) 
Then as in Lemma 6.1 of \cite{BM1} we have
\be \label{e:3pg}
  \bP( z \in \beta | \al ) = \wt G_D(x_0,z)
  \bP\big( \sL Y' \cap X' =\emptyset, \sL Y' \subset Q(x_0,m) \big). 
\ee
Due to Lemma \ref{L:wtG}, it remains to show that the probability
on the right hand side is bounded away from $0$.
We will in fact prove the stronger statement:
\eqn{e:avoid}
{ \bP\big( Y' \cap X' =\emptyset, Y' \subset Q(x_0,m) \big)
  \ge c > 0. } 
This result is not surprising, since two independent SRW in $\bZ^d$
(with $d\ge 5$) intersect with probability strictly less than 1.

Let us denote $A_z = Q(z,m/16)$, $B = Q(x_0,m)$ and $B' = Q(x_0,m/16)$.
Note that $Y'$ starts at $x_0$ and ends at $z$. 
We decompose $Y'$ into four subpaths, defined below, and give separate
estimates for these subpaths that together will imply the lower bound
on the probability in \eqref{e:avoid}. We define:
\eqnsplst
{ Y'_1
  = \sE^F_{\pd B'} (Y') \qquad
  Y'_2
  = \sE^L_{\pd A_z} ( \sB^F_{\pd B'} (Y') ) \qquad
  Y'_3
  = \sB^L_{\pd A_z} (Y'). }
That is, $Y'_1$ ends at the first exit from $B'$, 
$Y'_3$ begins at the last entrance to $A_z$ and $Y'_2$ is
the portion in between. We let $y_1 = Y'_1(|Y'_1|) = Y'_2(0)$ and
$y_2 = Y'_2(|Y'_2|) = Y'_3(0)$. 
We further decompose $Y'_2$ into the pieces:
\eqnsplst
{ Y'_{2,1}
  = \sE^F_{y_2} (Y'_2) \qquad
  Y'_{2,2}
  = \sB^F_{y_2} (Y'_2). }
That is, $Y'_{2,1}$ is the piece from $y_1$ to the first hit on $y_2$,
and $Y'_{2,2}$ is the remaining loop at $y_2$.
Observe that conditional on $y_1$ and $y_2$, the paths
$Y'_1, Y'_{2,1}, Y'_{2,2}, Y'_3$ are independent.
We now state our estimates for each piece.
Our notation will assume that $x_0 \in \bH_n$; trivial modification 
can be made when this is not the case.

\ms

\emph{Claim 1.} 
There is constant probability that $Y'_1$ exits 
$B'$ on the right hand face. That is,
we have $\bP ( y_1 \in \cR_{n,m/16} ) \ge c > 0$,
where $\cR_{n,m/16} = \bH_{n+m/16} \cap Q(x_0,m/16)$.

\emph{Proof of Claim 1.} 
Using Lemma \ref{L:ubh-tauD} we have
\eqnsplst
{ \bP ( y_1 \in \cR_{n,m/16} )
  &= \frac{\bP^{x_0} ( \wt X(\wt \tau_{B'}) \in \cR_{n,m/16},\, \wt T_{z} < \wt \tau_D )}{\bP^{x_0} ( \wt T_{z} < \wt \tau_D )} \\
  &\ge \frac{\wt G_D(z,z)}{\wt G_D(x_0,z)} \, \bP^{x_0} ( \wt X(\wt \tau_{B'}) \in \cR_{n,m/16} ) \, 
    \min_{w \in \cR_{n,m/16}} \bP^{w} ( \wt T_{z} < \tau_D ) \\
  &\ge c \min_{w \in \cR_{n,m/16}} \frac{\wt G_D(w,z)}{\wt G_D(x_0,z)} 
  \ge c. }

\ms 

In the next three claims we will use the notation
$B'' = x_0 + ([0, z_1 + m/32] \times [-m,m]^{d-1}) \cap \bZ^d$.

\ms

\emph{Claim 2.} 
There is constant probability that the following six events occur:\\
(i) $Y'_3$ starts on the left hand face of $A_z$;\\
(ii) $Y'_3 \subset z + ([-m/16,m/32] \times [-m/16,m/16]^{d-1}) \cap \bZ^d$;\\
(iii) $X'$ exits $A_z$ on the right hand face;\\
(iv) $X' \cap A_z \subset z + ([-m/32,m/16] \times [-m/16,m/16]^{d-1}) \cap \bZ^d$;\\
(v) $Y'_3 \cap (X' \cap A_z) = \es$;\\
(vi) $\sB^F_{\pd A_z} (X')$ is disjoint from $B''$.

\emph{Proof of Claim 2.}
Let $\wt S^{z}$ be the process defined as $S^{z}$ conditioned to
hit on $x_0$ before $T_{\al \setminus \{x_0\}} \wedge \tau_D$.
The time-reversal of $Y'$ has the law of $\wt S^{z}$. 
Therefore, the time-reversal of $Y'_3$ has the law of $\sE^F_{\pd A_z} (\wt S^{z})$.
The proof of Lemma \ref{L:sep} (Separation Lemma), shows that 
for independent simple random walks $S^{z}$ and $S'^{z}$ there is
probability $\ge c > 0$ that the analogues of the events 
(i)--(v) all hold. An application of the 
Harnack principle then shows that in fact (i)--(v) hold
with constant probability. 

It is left to show that conditionally on (i)--(v), we also 
have (vi) with constant probability. Since $X'$ is
$S$ conditioned on $T_\al > \tau_D$, this can be proved 
in the same way as Lemma \ref{L:ubh-ext}. For this 
we merely have to replace $Q(x_0,m/2)$ in that lemma
by $B''$, and make straightforward adjustments. Hence Claim 2 follows.

\ms 

\emph{Claim 3.}
Conditional on $y_1$ being in the right hand face of $B'$ and 
$y_2$ being in the left hand face of $A_z$, there is constant probability that 
$Y'_{2,1} \subset B''$.

\emph{Proof of Claim 3.}
Condition on $y_1$ and $y_2$. Then $Y'_{2,1}$ has the
law of $S^{y_1}$ conditioned to hit on $y_2$ before 
$T_\al \wedge \tau_D$ (stopped at the first hit on $y_2$).
Since $y_1$ and $y_2$ are at least distance $c m$ from 
the boundary of $B''$, such a path has constant 
probability to stay inside $B''$. (One way to see this
is to use an argument similar to that of Lemma \ref{L:ubh-ext},
where we let $R$ count the number of crossings by the walk 
from $Q(z,m/64)$ to $\pd B''$ before time 
$T_z \wedge T_\al \wedge \tau_D$.) Hence the claim
follows.

\ms

\emph{Claim 4.}
Conditional on $y_2$ being in the left hand face of $A_z$,
there is constant probability that $Y'_{2,2} \subset B''$.

\emph{Proof of Claim 4.}
Condition on $y_2$. The probability that $Y'_{2,2}$ consists 
of a single point is 
$G_{D \setminus \al} (y_2,y_2)^{-1} \ge G(y_2,y_2)^{-1} \ge c > 0$.

\ms 

When all the events in Claims 1--4 occur, the event in \eqref{e:avoid}
occurs. Hence the Lemma follows.

\ms
An application of Lemmas \ref{L:Mub} and \ref{L:ptlb} and the
one-sided Chebyshev inequality give the following corollary.

\begin{cor} \label{C:HAlb}
When $d \ge 5$, there exists a constant $c_0>0$ such that
$$ \bP( H_A(\beta) \ge c_0 m^2| \al ) \ge c_0. $$
\end{cor}

\begin{prop} 
\label{P:len-lb}
Assume $d \ge 5$. Let $N \ge 1$ and $Q_{4N} \subset D \subset \bZ^d$. 
Let $L = \sL \sE^F_{\pd D} S$ be a loop erased walk
from $0$ to $\pd D$, and $M_N = |\sE^F_{\pd_i Q_N} L|$
be the number  of steps in $L$ until its first hit on $\pd_i Q_N$.
Then for all $\lambda > 0$ we have
\be \label{e:Llb}
 \bP( M_N < \lam N^2) \le C \exp( -c \lam^{-1} ). 
\ee
\end{prop}

\proof 
Suppose $k \ge 1$ and $m \ge 4$ such that $N/2 \le k m < N-m$.
For $j=1, \dots, k$ let 
$$ \al_j = \sE^F_{\pd_i Q(0, jm) } L, \q \sF_j = \sigma( \al_j). $$
Let $Y_j = \al_j ( |\al_j|)$ be the last point in $\al_j$, and
$$ \beta_j = \sE^F_{\pd_i Q(Y_j,m)}( \sB^F_{\pd_i Q(0, jm) }L) $$
be the path $L$ between $Y_j$ and its first hit after $Y_j$ on
$\pd_i( Y_j, m)$. We have
$$ M_N \ge \sum_{i=1}^k |\beta_j|, $$
Let $G_j = \{ |\beta_j| < c_0 m^2 \}$; then by Corollary \ref{C:HAlb}
$$ \bP( G_j  | \sF_j) \le 1 - c_0. $$
Therefore, $M_N$ stochastically dominates a sum of $k$
independent random variables that take the values $c_0 m^2$ and $0$
with probabilities $c_0$ and $1 - c_0$, respectively. Hence
\begin{align*}
\bP( M_N \le (1/2) k c_0^2 m^2 ) &\le C \exp ( - c k ).
\end{align*}
We now take $k \asymp \lambda^{-1}$ and $m \asymp \lambda N$ 
and we obtain \eqref{e:Llb}.
\qed

In the following theorem, we obtain a lower bound on the length
of paths in the USF. We define the event:
\be \label{e:Fndef}  
F(y,x,n) = \left\{ \text{$T_x[S^y] < \infty$ and $|\sL \sE^F_x (S^y)| \le n$} \right\}. 
\ee

\begin{thm} \label{T:Ass2}
For every $x, y \in \bZ^d$ we have
\be \label{e:clb}
 \bP(F(y,x,n))
  \le C (1 + |x-y|)^{2-d} \exp \left[ - c \frac{|x-y|^2} {n} \right]. 
\ee
\end{thm}

\proof 
For notational convenience, we assume $y = 0$ (otherwise translate $x, y$ by $-y$).
If $|x|^2/n \le 1$ then the term in the exponential in \eqref{e:clb}
is of order 1, so
$$ \bP( F(0,x,n)) \le \bP( T_x < \infty) \le (1+|x|)^{2-d} 
\le e^c (1+|x|)^{2-d} e^{- c |x|^2 / n}.  $$

Now assume $|x|^2 > n$, and let $N= ||x||_\infty/4$, and  $Q=Q(0,N)$.
Let $X'$ be $S$ conditioned on $\{ T_x < \infty\}$. Then if
$h(z) = \bP^z( T_x[S] < \infty)$, we have $h(z) \asymp N^{2-d}$ on $Q(0,N)$, and
thus the processes $S$ and $X'$ have comparable laws inside $Q(0,N)$.
The explicit law of a section of the loop erased random path 
given in \cite{Law99} (see also (5) in \cite{Mas})
then implies that the loop erasures of $S$ and $X'$ 
also have comparable laws inside $Q$.

Let 
\be
 F_1(x,n) = \left \{ |\sE^F_{\pd_i Q} ( \sL \sE^F_x S) | \le n,  
  T_x< \infty \right \}. 
\ee
Thus $F(0,x,n) \subset F_1(x,n)$. Then
\begin{align}
\nn
 \bP( F(0,x,n)) &\le  \bP( F_1(x,n)) \\
\nn
&= \bP( | \sE^F_{\pd_i Q} \sL(\sE^F_x S) | \le n | T_x<\infty) \, \bP(T_x < \infty)\\
\nn
&\le C |x|^{2-d} \bP( |\sE^F_{\pd_i Q} \sL(  \sE^F_x X' )| \le n  ) \\ 
\nn 
&\le C |x|^{2-d} \bP( |\sE^F_{\pd_i Q} \sL(  \sE^F_x S)| \le n  ).
\end{align}
Taking $n=\lam N^2$, so that $\lam^{-1} \ge c |x|^2 n^{-1}$, 
and using Proposition \ref{P:len-lb} completes the proof. \qed

\section{Upper bound on $| B_\sU(0,n) |$} \label{sec:ub}

Recall that $\sU(x)$ is the component of the USF containing $x \in \Z^d$.
It is well-known \cite[Theorem 4.2]{Pem91} that 
for $d \ge 5$ and $x \not= y \in \Z^d$ we have 
\eqn{e:cnctd-bnd}
{ c |x - y|^{4-d} 
  \le \bP ( y \in \sU(x) )   
  \le C |x - y|^{4-d}. }
A corollary of this bound is that the volume of $\sU_0 \cap B(r)$ 
grows as $r^4$ in expectation. Our main result in the previous section, 
Theorem \ref{T:Ass2}, is a variant of the upper bound in \eqref{e:cnctd-bnd} 
that gives control over the length of the path connecting $x$ and $y$. 
Since that bound was formulated in terms of a single LERW, the exponent $4-d$ 
changes to $2-d$. In this section we extend Theorem \ref{T:Ass2} to control 
the volume of balls in the intrinsic metric.

\begin{theorem}
\label{thm:B_U-moments}
Assume $d \ge 5$, and let $\sU = \sU_{\bZ^d}$. There exists a constant $C_1$ such that 
for all $k \ge 0$ we have
\be \label{e:ubk}
 \bE \big( |B_\sU(0,n)|^k \big) \le C_1^k k! n^{2k}. 
\ee
Hence there are constants $c_1 > 0$ and $C_2$ such that 
\be \label{e:ub-exp}
 \bP ( |B_\sU(0,n)| \ge \lam n^2 ) \le C_2 e^{-c_1 \lam}, 
   \quad \lambda > 0,\, n \ge 1.
\ee
\end{theorem}

\begin{proof}
The bound \eqref{e:ub-exp} follows easily from \eqref{e:ubk} using Markov's
inequality and the power series for $e^x$.

We prove \eqref{e:ubk} by induction on $k$. The case $k = 0$ holds trivially.
We fix $k \ge 1$ and $y_1, \dots, y_k \in \Z^d$, and
estimate the probability
\eqnst
{ \bP \big( y_1, \dots, y_k \in B_\sU(0,n) \big). }
This can be done similarly to the ``tree-graph inequalities'' known
in percolation \cite{AN84}. To facilitate notation, we write $y_0 = 0$. On the event 
$y_1, \dots, y_k \in \sU_0$ consider the minimal subtree 
$T(y_0, \dots, y_k) \subset \sU_0$ that contains 
the vertices $y_0, \dots, y_k$. This tree is finite. 
Since $\sU_0$ has one end \cite{BLPS}, \cite{LP:book}, 
there is a unique infinite path in $\sU_0$, whose only vertex in $T(y_0, \dots, y_k)$
is its starting vertex.
Let us write $T(y_0, \dots, y_k, \infty)$ for the infinite 
subtree of $\sU_0$ obtained by adding this infinite path to
$T(y_0, \dots, y_k)$. 

Now let us consider the ``topology'' of 
$T(y_0, \dots, y_k, \infty)$. In the case $k = 1$, it is
easy to see that there exists a vertex $z_1 \in T(y_0, y_1, \infty)$
such that the paths $T(y_0,z_1)$, $T(y_1,z_1)$ and $T(z_1,\infty)$ 
(some of which may degenerate to a single vertex) are edge-disjoint.
In the general case $k \ge 1$, we have $k$ ``branch points''
$z_1, \dots, z_k$. We use a fixed rule for indexing
the $z_i$'s, in requiring that for every $i \ge 1$
the path $T(y_i,z_i)$ is edge-disjoint from
$T(y_0, \dots, y_{i-1}, \infty)$. See Figure \ref{fig:taus}.
\begin{figure}
\includegraphics{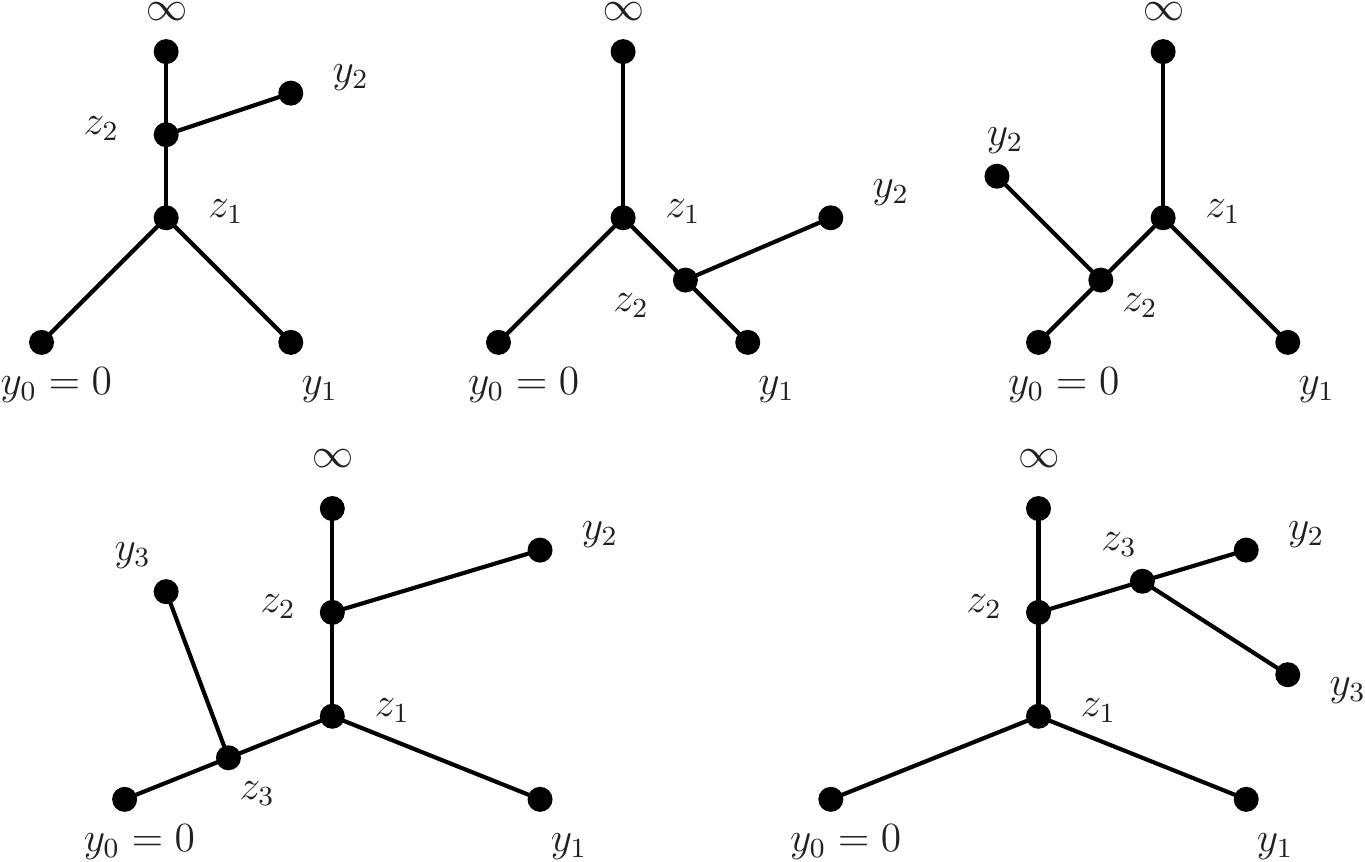}

\caption{\label{fig:taus} 
 All three labelled tree graphs with $k = 2$, 
and two of the five possible labelled tree graphs with $k = 3$.}
\end{figure}

We can formalize the construction via the following 
recursive procedure. Let $\cT(0)$ denote the set containing 
the unique tree with vertex set $\{ 0, \infty \}$. Assume
that the collection $\cT(k-1)$ of trees with vertex set 
$\{ 0, \dots, k-1 \} \cup \{ \infty \} \cup \{ \bar{1}, \dots, \overline{k-1} \}$
has been defined for some $k \ge 1$. Let $\cT(k)$ denote 
the collection of trees with vertex set 
$\{ 0, \dots, k \} \cup \{ \infty \} \cup \{ \bar{1}, \dots, \bar{k} \}$
that can be obtained in the following way. Pick some
$\tau' \in \cT(k-1)$, and pick one of the edges of $\tau'$.
Split this edge into two by introducing a new vertex 
$\bar{k}$ on the edge, and add the new edge $\{ k, \bar{k} \}$
to $\tau'$. 
It is easy to see that any $\tau \in \cT(k)$ has the following
properties (see Figure \ref{fig:taus}):
\begin{itemize}
\item[(i)] $\deg_\tau(\infty) = 1 = \deg_\tau(y_i)$, 
$i = 0, \dots, k$.
\item[(ii)] $\deg_\tau(\bar{i}) = 3$, $i = 1, \dots, k$.
\end{itemize}

With the above definitions, the event
$\{ y_1, \dots, y_k \in \sU_0 \}$ implies that there exist
$z_1, \dots, z_k \in T(y_0, \dots, y_k, \infty)$ and 
$\tau \in \cT(k)$ such that $T(y_0, \dots, y_k, \infty)$ is the
edge-disjoint union of paths $T(\varphi(r),\varphi(s))$, where
$\{ r, s \} \in E(\tau)$, and $\varphi : V(\tau) \to \Z^d \cup \{ \infty \}$ 
is defined by 
\eqn{e:def-varphi}
{ \begin{cases} 
  \varphi(i) = y_i & i = 0, \dots, k; \\
  \varphi(\infty) = \infty; \\
  \varphi(\bar{i}) = z_i & i = 1, \dots, k.
  \end{cases} }
Note that the choice of $\tau$ is not unique, due to possible 
coincidences between the vertices $y_0, \dots, y_k, z_1, \dots, z_k$. 
We neglect the overcounting resulting from this, for an upper bound.

If the additional restriction $d_\sU(0,y_i) \le n$, $i = 1, \dots, k$
is in place, we must also have $d_\sU(\varphi(r), \varphi(s)) \le n$
for all $\{ r, s \} \in E(\tau)$ such that 
$r, s \not= \infty$. We define the event
\eqnsplst
{ &E(y_1, \dots, y_k, z_1, \dots, z_k, \tau, n) \\
  &\qquad= \left\{ \parbox{8.5cm}{$T(y_0, \dots, y_k, \infty) 
      = \cup_{\{ r, s \} \in E(\tau)} 
      T(\varphi(r),\varphi(s))$ as an edge-disjoint union and
      $d_\sU(\varphi(r),\varphi(s)) \le n$ for all 
      $\{ r, s \} \in E(\tau)$ such that 
      $r, s \not= \infty$} \right\}. }  
Considering all possible choices of $\tau$ and $z_1, \dots, z_k$,
we get 
\eqnsplst
{ \bE \big( | B_\sU(0,n) |^k \big) 
  &= \sum_{y_1, \dots, y_k \in \Z^d} 
     \bP \big( y_1, \dots, y_k \in B_\sU(0,n) \big) \\
  &\le \sum_{\tau \in \cT(k)} \sum_{y_1, \dots, y_k \in \Z^d} 
     \sum_{z_1, \dots, z_k \in \Z^d}
     \bP \big( E(y_1, \dots, y_k, z_1, \dots, z_k, \tau, n) \big). }

We use Wilson's algorithm \cite{W,LP:book} to replace the complicated
event $E(y_1, \dots )$ by a slightly larger event that is easier to handle.
For this, enumerate the edges of $\tau$ as
\eqnst
{ \{ r_0, s_0 \}, \{ r_1, s_1 \}, \dots, \{ r_{2k}, s_{2k} \}, }
where the labelling is chosen in such a way that the following two
properties are satisfied (see Figure \ref{fig:enum}(a)):
\begin{figure}
\includegraphics{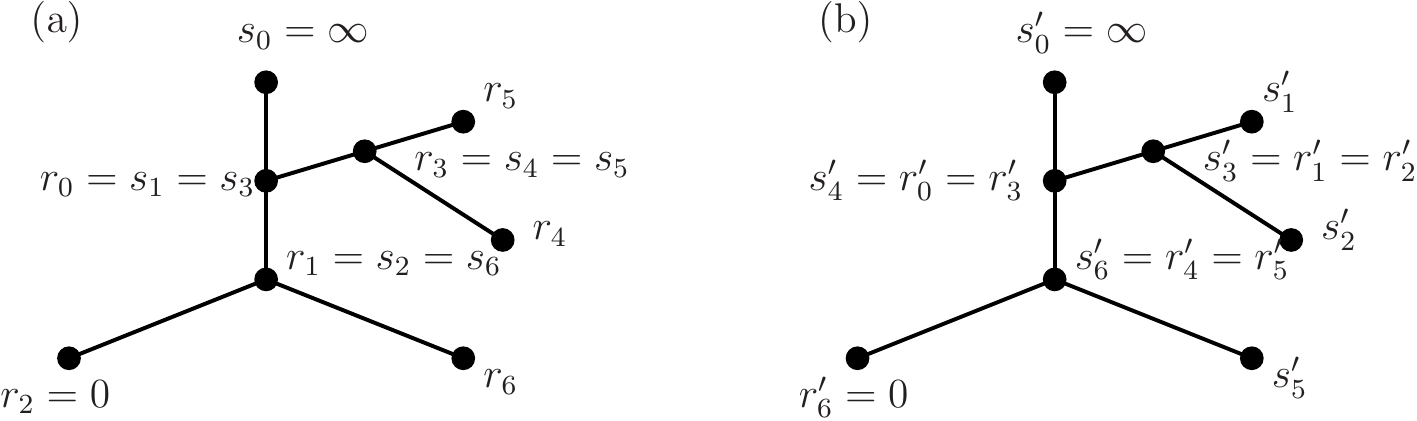}
\caption{\label{fig:enum}(a) A possible enumeration of edges
for the application of Wilson's method.
(b) A possible enumeration of edges for performing the
summations using \eqref{e:conv-bnd} in the order
$j = 1, 2, \dots, 2k$. Summing over the spatial location
$\varphi(s'_1)$ eliminates the factor involving the edge
$\{ s'_1, r'_1 \}$. Following this, it is possible to sum 
over $\varphi(s'_2)$, etc.}
\end{figure}
\begin{itemize}
\item[(a)] $s_0 = \infty$.
\item[(b)] For every $j = 1, \dots, 2k$, the set of edges
$\{ \{ r_\ell, s_\ell \} : \ell = 0, \dots, j-1 \}$ spans
a subtree of $\tau$, and $s_j$ is a vertex of this subtree.
\end{itemize}
Using Wilson's method with random walks started at
$\varphi(r_0), \dots, \varphi(r_{2k})$, we see that 
\eqn{e:E-incl}
{ E(y_1, \dots, y_k, z_1, \dots, z_k, \tau, n)
  \subset \bigcap_{j=1}^{2k} F(\varphi(s_j),\varphi(r_j),n). }
Here $F(\cdot, \cdot, n)$ are the events defined in \eqref{e:Fndef}.  
Importantly, the events on the right hand side are independent. 
Theorem \ref{T:Ass2}  and the inclusion \eqref{e:E-incl} imply that
\eqnspl{e:E-bnd}
{ &\bP \big( E(y_1, \dots, y_k, z_1, \dots, z_k, \tau, n) \big) \\
  &\qquad\le \prod_{j=1}^{2k} 
      C (1 + |\varphi(s_j) - \varphi(r_j)|)^{2-d}
      \exp \left[ - c \frac{|\varphi(s_j) - \varphi(r_j)|^2}{n} \right]. }
It remains to estimate the sum of the right hand side 
of \eqref{e:E-bnd} over all choices of the $y_i$'s and
$z_i$'s. For this it will be convenient to use a different
enumeration of $E(\tau)$. Suppose that  
\eqnst
{ \{ r'_0, s'_0 \}, \{ r'_1, s'_1 \}, \dots,
  \{ r'_{2k}, s'_{2k} \} }
satisfies the following properties (see Figure \ref{fig:enum}(b)).
\begin{itemize}
\item[(a')] $s'_0 = \infty$ and $r'_{2k} = 0$.
\item[(b')] For every $j = 1, \dots, 2k$ the set
$\{ \{ r'_\ell, s'_\ell \} : \ell = j, \dots, 2k \}$
induces a connected subtree of $\tau$, and $s'_j$ is a 
leaf of this subtree.
\end{itemize}
For ease of notation, let us write 
$u_j = \varphi(r'_j)$ and $w_j = \varphi(s'_j)$.
With the new enumeration the right hand side of 
\eqref{e:E-bnd} takes the following form:
\eqnspl{e:E-bnd2}
{ &\bP \big( E(y_1, \dots, y_k, z_1, \dots, z_k, \tau, n) \big) \\
  &\qquad\le \prod_{j=1}^{2k} 
      C (1 + |w_j - u_j|)^{2-d}
      \exp \left[ - c \frac{|w_j - u_j|^2}{n} \right]. }
Note again that the $w_j$'s and $u_j$'s are 
$z_i$'s and $y_i$'s, determined implicitly 
by $\tau$. Importantly, property (b') of the enumeration 
implies that if $w_j = \varphi(s'_j) = z_i$ for some $i, j$,
then the variable $z_i$ does not occur in the product
\eqnst
{ \prod_{\ell=j+1}^{2k} 
      C (1 + |w_j - u_j|)^{2-d}
      \exp \left[ - c \frac{|w_j - u_j|^2}{n} \right]. }
Similar considerations apply if $w_j = \varphi(s'_j) = y_i$
for some $i, j$. The summation over 
$y_1, \dots, y_k$ and $z_1, \dots, z_k$ 
can be accomplished by the following lemma.

\begin{lemma}
\label{lem:conv-bnd}
For any $u \in \Z^d$, we have
\eqn{e:conv-bnd}
{ \sum_{w \in \Z^d} (1 + |w - u|)^{2-d}
    \exp \left[ - c \frac{|w - u|^2}{n} \right]
  \le C n. }
\end{lemma}

We apply Lemma \ref{lem:conv-bnd} successively 
to the factors with $j = 1, \dots, 2k$ on the 
right hand side of \eqref{e:E-bnd2}. See Figure \ref{fig:enum}(b)
for an example of how the edges of $\tau$ are 
successively removed by the summations.
We obtain
\eqn{e:B_U-moments}
{ \bE \big( | B_\sU(0,n) |^k \big) 
  \le \sum_{\tau \in \cT(k)} (C n)^{2k}. }
Since the number of trees in $\cT(k)$ is 
$1 \cdot 3 \cdot \cdots (2k-1) \le 2^k k!$, this proves \eqref{e:ubk}.
\qed
\end{proof}

\begin{remark}
The statements of Theorem \ref{thm:B_U-moments} still hold, with 
essentially the same proof, when $\sU = \sU_D$, with any $D \subset \bZ^d$. 
Note that $\sU_0$ still has one end. This follows from \cite[Proposition 3.1]{LMS},
and the fact that the component of $0$ under the measure 
$\mathsf{WSF}_o$ in the domain $D$ is stochastically smaller
then it is in $\bZ^d$. Therefore, a decomposition into events
$E(y_1,\dots,y_k,z_1,\dots,z_k,n)$ still holds (with $\sU = \sU_D$), 
where now all vertices are in $D$. The inclusion \eqref{e:E-incl} still holds,
with the events $F$ having the same meaning as before. This allows
to bound the summations in exactly the same way as in $\bZ^d$.
\end{remark}

\section{Lower bounds on volumes} \label{sec:vol-lb}

In this section we return to the setup of Section \ref{sec:lew},
in order to give a lower bound on the volume of $\sU_0$. 
We first estimate the number of vertices of $\sU_0$ in shells 
$Q_{n+m} \setminus Q_n$. Recall that $Q_N \subset D \subset \bZ^d$, 
and $n, m$ satisfy $16 \le n < n+m \le N$, with $m \le n/8$. 
We have $L = \sL ( \sE^F_{D^c} (S))$, $\al = \sE^F_{ \pd_i Q_n } L$, 
and $x_0 \in \pd_i Q_n$ is the endpoint of $\al$. 
The remaining piece of $L$ is $L' = \sB^F_{\pd_i Q_n}L$, and 
$\beta = \sE^F_{\pd_i Q(x_0,m)} L'$.
See Figure \ref{fig:boxes-cycle-pop}.
\begin{figure}
\centerline{\includegraphics{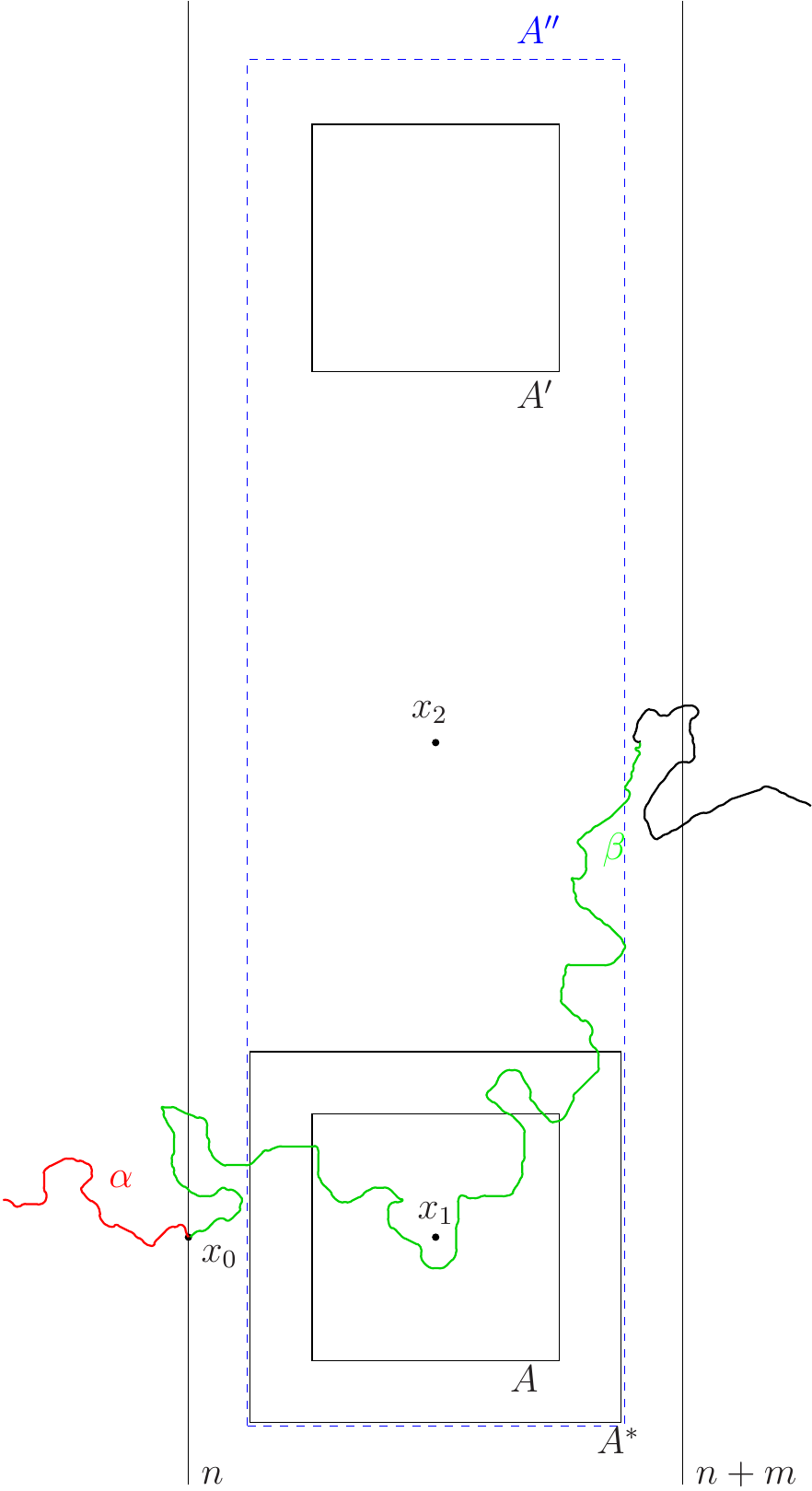}}
\caption{\label{fig:boxes-cycle-pop} Boxes for the cycle popping argument.}
\end{figure}
Recall that when $x_0 \in \bH_n$, we defined 
$A = A(x_0) = Q(x_0 + (m/2) e_1, m/4)$ and $x_1 = x_0 + (m/2) e_1$,
with appropriate rotations applied when $x_0$ was on a different 
face of $Q_n$. We will now also need a point $x_2 \in Q_{n+m} \setminus Q_n$
of order $m$ away from $A$, and further boxes contained in 
$Q_{n+m} \setminus Q_n$ that we define as follows. 
If $x_0 \in \bH_n$ and the second coordinate of $x_0$ is negative, let 
\eqnspl{e:A'-A''-def}
{ x_2
  &= x_1 + m e_2 \\
  A'
  &= A'(x_0)
  = Q( x_1 + 2 m e_2, m/4 ) \\
  A''
  &= A''(x_0)
  = x_1 + [-3m/8,3m/8] \times [-m,3m] \times [-m,m]^{d-2} \cap \bZ^d. }
If $x_0 \in \bH_n$ and the second coordinate of $x_0$ is positive,
we replace $e_2$ by $-e_2$ and $[-m,3m]$ by $[-3m,m]$. If 
$x_0$ is on a different face of $Q_n$, we replace $e_1$ and $e_2$
by two other suitable unitvectors.

The key technical estimate is to show that $\beta \cap A$ has capacity
of order $m^2$ with probability bounded away from $0$, which we
do in the next section.

\subsection{A capacity estimate}

Let $S^{x_2}$ be a random walk with $S^{x_2}(0) = x_2$, independent of $S$, 
$\tilde{X}$, etc.

\begin{proposition} \label{P:capacity}
Assume $N \ge 1$, $Q_{4N} \subset D \subset \bZ^d$, and the setup of Section \ref{sec:lew}.\\
(a) There exists $c_1 = c_1(d) > 0$ such that 
\eqnst
{ \bP \big( \text{$S^{x_2}$ hits $(A \cap \beta)$} \,\big|\, \alpha \big)
  \ge c_1 m^{4-d}. } 
(b) We have
\be \label{e:capest1}
 \bP \big( c m^2 \le \Capac (A \cap \beta) \le C_1 m^2 \,\big|\, \alpha \big)
  \ge c > 0. 
\ee
\end{proposition}

\begin{proof}
(a) For ease of notation, we omit the conditioning on $\alpha$.
Let 
\eqnst
{ U  := \sum_{ z \in A} I [z \in \beta] 
     I [\text{$S^{x_2}$ hits $z$}], }
 so that    
$$    \bP \big( \text{$S^{x_2}$ hits $(A \cap \beta)$} \big)
  = \bP ( U > 0 ). $$
Using Lemma \ref{L:ptlb}, we have
\eqnsplst
{ \bE ( U ) 
  = \sum_{z \in A} \bP ( z \in \beta ) 
    \bP \big( T_z[S^{x_2}] < \infty \big) 
  \ge c m^d m^{2-d} m^{2-d}
  = c m^{4-d}. } 
On the other hand, 
\eqn{e:U^2}
{ \bE \big( U^2 \big)
  = \sum_{x, y \in A} \bP ( x, y \in \beta ) \,
    \bP \big( T_x[S^{x_2}] < \infty,\, T_y[S^{x_2}] < \infty \big). }
Since the process $\wt X$ generating $L'$ must pass through
$\pd A^*$ in order for the event $x, y \in \beta$ to occur,
we have
\eqnsplst
{ \bP ( x, y \in \beta )
  &\le \max_{z \in \pd A^*} [ \wt G_D(z,x) \wt G_D(x,y) + \wt G_D(z,y) \wt G_D(y,x) ] \\
  &\le C m^{2-d} G(x,y). }
For the other term in the right hand side of \eqref{e:U^2} we have
\eqnsplst
{ \bP \big( T_x[S^{x_2}] < \infty,\, T_y[S^{x_2}] < \infty \big)
  &\le [ G(x_2,x) G(x,y) + G(x_2,y) G(y,x) ] \\
  &\le C m^{2-d} G(x,y). }
Since $d \ge 5$, we have $\sum_{x, y \in A} G(x,y)^2 \le C m^d$, which gives
$\bE \big( U^2 \big) \le C m^{4-d}$.

The Paley-Zygmund inequality then gives 
\eqnst
{ \bP \big( \text{$S^{x_2}$ hits $(A \cap \beta)$} \big)
  = \bP ( U > 0 )
  \ge \frac{\bE ( U )^2}{\bE \big( U^2 \big)}
  \ge c m^{4-d}. }

(b)   Since $\Capac(A \cap \beta) \le C |A \cap \beta|$,
and $m^{2-d} \Capac(A \cap \beta) \asymp 
\bP \big( T_{A \cap \beta}[S^{x_2}] < \infty \,\big|\, \beta \big)$,
combining (a) with Lemma \ref{L:Mub} gives (b). 
\qed  
\end{proof}

Assume now, similarly to Proposition \ref{P:len-lb}, that 
$k \ge 1$ and $m \ge 4$ such that $N/2 \le k m < N-m$.
Recall that for $j=1, \dots, k$ we denote 
$\al_j = \sE^F_{\pd_i Q(0, jm) } L$. Let $Y_j = \al_j ( |\al_j|)$ 
be the last point in $\al_j$, and
$\beta_j = \sE^F_{\pd_i Q(Y_j,m)}( \sB^F_{\pd_i Q(0, jm) }L)$
be the path $L$ between $Y_j$ and its first hit after $Y_j$ on
$\pd_i( Y_j, m)$. Let $Y_{j,1}$ and $Y_{j,2}$ be the points $x_1$ and 
$x_2$ defined with respect to $x_0 = Y_j$, respectively.
Define the following event, measurable with respect to $L$:
\eqn{e:G-event}
{ G(c_1, c_2, C_1)
  = \left\{ \parbox{9cm}{there are at least $c_2 k$ indices
       $j$ with $1 \le j \le k$ such that 
       $\bP \big( T_{A(Y_j) \cap \beta_j}[S^{Y_{j,2}}] < \infty \,\big|\, L \big) 
       \ge c_1 m^{4-d}$ and $|Q(Y_j,m) \cap \beta_j| \le C_2 m^2$} \right\}. }
Proposition \ref{P:capacity} and an argument similar to that of 
Proposition \ref{P:len-lb} gives the following corollary.

\begin{corollary}
\label{C:enough boxes} 
Under the assumptions of Proposition \ref{P:capacity}, 
there exist $c_1, c_2 > 0$ and $C_2$ such that we have
\eqnspl{e:good boxes}
{ \bP \left[ G(c_1, c_2, C_2)  \right] 
  \ge 1 - \exp( -c k ). } 
\end{corollary}

\begin{remark}
\label{R:cond-exit}
We note the following minor extension of Corollary \ref{C:enough boxes}.
Assuming still that $Q_{4N} \subset D$, let $w \in \partial D$
be fixed, condition $S$ to exit $D$ at $w$, and let $L' = \sL (\sE^F_{D^c} S)$ 
be the loop-erasure. Masson \cite{Mas} proves that the law of 
$\sE^F_{Q_N^c} L'$ is comparable, up to constants factors, to the 
law of $\sE^F_{Q_N^c} L$. Since the event $G(c_1,c_2,C_2)$ 
is measurable with respect to $\sE^F_{Q_N^c} L$, the statement
of the corollary follows also for $L'$.
\end{remark}

\subsection{Lower bound on $| Q_N \cap \sU_0 |$}
\label{ssec:cycle-pop}

We continue with the setup of the previous section.
Our argument will use the cycle popping idea of Wilson \cite{W};
see also \cite{LP:book}.

\begin{theorem}
\label{T:cyclepop}
Assume $N \ge 1$, $Q_{4N} \subset D \subset \bZ^d$, and 
let $\sU = \sU_D$. There exist constants $C, c$, such that 
\eqnst
{ \bP \big( |Q_N \cap \sU_0| \le \lambda N^4 \big)
  \le C \exp ( - c \lambda^{-1/3} ). }
\end{theorem}

\begin{proof}
Condition on $L$, and assume that the event \eqref{e:G-event} occurs. 
Let $J$ be the set of indices $1 \le j \le k$ (a $\sigma(L)$-measurable 
random set) satisfying the requirements in this event. 
For each $j \in J$, let 
\eqnst
{ A'(j) 
  = A'(Y_j) \qquad
  A''(j)
  = A''(Y_j). }
The definitions of $A'$ and $A''$ made in \eqref{e:A'-A''-def}
ensure that $A''(j)$, $j \in J$ are disjoint. 

We will need two coupled collections of stacks.
Associate to each $z \in (\cup_{j \in J} A''(j)) \setminus L$ 
a stack of arrows, and let us call these $\mathsf{Stacks\ I}$.
For each $j \in J$ and each $z \in A''(j) \cap L \setminus \beta_j$,
pick a new independent stack leaving the rest of the stacks 
unchanged. Call this second collection of stacks 
$\mathsf{Stacks\ II}$. In both $\mathsf{Stacks\ I}$ and 
$\mathsf{Stacks\ II}$, and for every $j \in J$,
pop all cycles that are entirely contained in $A''(j)$.
That is, if a cycle starts in $A''(j)$, but part of it
lies outside $A''(j)$, we do not pop it.
It is important to note that the order of popping cycles
is irrelevant for determining the final configuration 
on the top of the stacks.

For each $j \in J$, let 
\eqnsplst
{ V^I_j 
  &= \left\{ y \in A'(j) : 
    \parbox{5.5cm}{cycle popping using $\mathsf{Stacks\ I}$
    reveals a path from $y$ to $L$} \right\} \\
  V^{II}_j
  &= \left\{ y \in A'(j) : 
    \parbox{6.3cm}{cycle popping using $\mathsf{Stacks\ II}$ \\
    reveals a path from $y$ to $A''(j) \cap \beta_j$} \right\} }
Note that $(V^I_j,V^{II}_j)_{j \in J}$ are conditionally 
independent, given $L$, $J$.

\begin{lemma}
\label{L:coupling}
We have $V^I_j \supset V^{II}_j$ for all $j \in J$. 
\end{lemma}

\begin{proof}
Let $y \in V^{II}_j$, and consider $\mathsf{Stacks\ II}$. Starting from $y$, 
follow the arrows in $\mathsf{Stacks\ II}$, until $A''(j) \cap \beta_j$ 
is hit. Removing cycles chronologically from this path 
pops some cycles entirely contained in $A''(j)$, and reveals a
path from $y$ to $A''(j) \cap \beta_j$. Now if we follow the
arrows in $\mathsf{Stacks\ I}$ instead, then the same arrows 
are used until the first time $L$ is hit. This guarantees
that a path from $y$ to $L$ is revealed, that does not 
leave $A''(j)$, and hence $y \in V^I_j$. 
\end{proof}

\begin{lemma}
\label{L:VII-lb}
Assume $d \ge 5$. For some $c_3 > 0$ we have
\eqnst
{ \bP \big( |V^{II}_j| \ge c_3 m^4 \,\big|\, L,\, j \in J \big)
  \ge c
  > 0. }
\end{lemma}

\begin{proof}
We estimate the first and second moments of $|V^{II}_j|$.

Fix $y \in A'(j)$. Following the arrows from $y$ in 
$\mathsf{Stacks\ II}$ we perform a random walk until 
either we exit $A''(j)$, or we hit $A''(j) \cap \beta_j$.
Therefore, 
\eqnspl{e:hit-from-y}
{ \bP \big( y \in V^{II}_j \,\big|\, L,\, j \in J \big)
  &= \bP \big( T_{A''(j) \cap \beta_j}[S^y] < \tau_{A''(j)}[S^y] 
    \,\big|\, L,\, j \in J \big) \\
  &\ge \bP \big( T_{A(j) \cap \beta_j}[S^y] < \tau_{A''(j)}[S^y] 
    \,\big|\, L,\, j \in J \big). }
The last expression is 
\eqn{e:hit-first}
{ \ge c \bP \big( T_{A(j) \cap \beta_j}[S^y] < \infty 
    \,\big|\, L,\, j \in J \big). }
(One way to see this is by an argument similar to 
that of Lemma \ref{L:ubh-ext}, where we let $R$ count
the number of crossings by the walk from a box 
$A^{**} \subset A''(j)$ to $\pd A''(j)$ before hitting $\beta_j \cap A(j)$,
where each face of $\pd A^{**}$ is at distance $m/16$ away 
from the corresponding face of $\pd A''(j)$.) 

The Harnack inequality and Proposition \ref{P:capacity} now implies,
after summing over $y$ in \eqref{e:hit-from-y}--\eqref{e:hit-first}, that 
\eqnst
{ \bE \big( |V^{II}_j| \,\big|\, L,\, j \in J \big)
  \ge c c_1 m^d m^{4-d}
  = c m^4. }  

We now bound the second moment of $|V^{II}_j|$.
If $x, y \in V^{II}_j$ occurs, then there exists 
a unique $w \in A''(j)$ with the property that cycle popping 
reveals three edge-disjoint paths:
one from $w$ to $A''(j) \cap \beta_j$, a second from 
$x$ to $w$ and a third from $y$ to $w$. (We allow to have 
$x = w$ or $y = w$ or both.) When this event 
happens with a fixed $w$, we can reveal the paths
by first following the arrows starting from $w$
until $A''(j) \cap \beta_j$ is hit, then following
the arrows starting from $x$ until $w$ is hit, then 
following the arrows starting from $y$ until $w$ 
is hit. This shows that 
\eqnspl{e:VII2ndm-ub}
{ &\bP \big( x, y \in V^{II}_j \,\big|\, L,\, j \in J \big) \\
  &\qquad \le \sum_{w \in A''(j)}
    \bP \big( T_{A''(j) \cap \beta_j}[S^w] < \infty 
    \,\big|\, L,\, j \in J \big) \,  
    \bP \big( T_w[S^x] < \infty \big) \, \bP \big( T_w[S^y] < \infty \big). }
Let 
$\tilde{A}(j) =  Q(Y_{j,1}, (3m/2))$, and note that 
$\partial \tilde{A}(j)$ has distance at least $c m$ from 
$A''(j) \cap \beta_j$, and also distance at least 
$c m$ from $A'(j)$. 
We estimate separately the cases:\\
(a) $w \in A''(j) \setminus \tilde{A}(j)$; and \\
(b) $w \in A''(j) \cap \tilde{A}(j)$.\\
The sum of the terms in the right hand side of \eqref{e:VII2ndm-ub}
corresponding to case (a) is at most:
\eqnsplst
{ &C m^{2-d} \Capac(A''(j) \cap \beta_j) 
    \sum_{w \in A''(j) \setminus \tilde{A}(j)}
    \sum_{x, y \in A'(j)} G(x,w) G(y,w) \\
  &\qquad \le C m^{2-d} m^2 m^2 m^2 m^d 
  = C m^8. }
The sum for case (b) is at most:
\eqnsplst
{ &C m^{2-d} m^{2-d} m^d m^d 
    \sum_{w \in A''(j) \cap \tilde{A}(j)} 
    \bP \big( T_{A''(j) \cap \beta_j}[S^w] < \infty \big) \\
  &\qquad \le C m^4 \sum_{w \in \tilde{A}(j)} 
    \bP \big( T_{A''(j) \cap \beta_j}[S^w] < \tau_{\tilde{A}(j)} \big) \\
  &\qquad \le C m^4 m^2 \Capac(A''(j) \cap \beta_j) \\
  &\qquad \le C m^8. }
Here the last line follows from $j \in J$ and Proposition \ref{P:capacity}.

The moment estimates for $|V^{II}_j|$ and the one-sided Chebyshev
inequality yield:
\eqnst
{ \bP \big( V^{II}_j \ge c m^4 \,\big|\, L,\, j \in J \big) 
  \ge c 
  > 0. }
This completes the proof of the Lemma.
\end{proof}

We can now complete the proof of Theorem \ref{T:cyclepop}.
Choose $k \asymp \lambda^{-1/3}$ so that $\lambda N^4 \asymp k m^4$. 
Then using Corollary \ref{C:enough boxes}, the conditional 
independence of $(V^{II}_j)_{j \in J}$, and Lemma \ref{L:coupling},
for a suitably small $c_4 > 0$ we have
\eqnsplst
{ \bP \big( |Q_N \cap \sU_0| \le \lambda N^4 \big) 
  &\le C \exp( - c k )
      + \bE \bigg( \bP \bigg( \parbox{4cm}{$V^{II}_j \ge c_3 m^4$ for less than $c_4 k$ indices $j \in J$} 
          \,\bigg|\, L \bigg) I[G(c_1,c_2,C_2)] \bigg) \\
  &\le C \exp ( -c \lambda^{-1/3} ). }
This completes the proof of the Theorem.
\end{proof}

\begin{theorem}
\label{T:UST-lb}
Assume $d \ge 5$ and let $\sU = \sU_{\bZ^d}$. There exist 
$c > 0$ and $C$ such that for all $\lambda > 0$ we have
\eqnst
{ \bP \big( |B_\sU(0,n)| \le \lambda n^2 \big)
  \le C \exp ( - c \lambda^{-1/5} ). }
\end{theorem}

For the proof of this theorem, we assume the setting of 
Proposition \ref{P:len-lb}, with $D = \Z^d$. 
Recall that $M_N = |\sE^F_{\partial_i Q_N} L|$. 

\begin{lemma}
\label{L:len-box-ub}
We have
\eqnst
{ \bE \big( M_N^k \big)
  \le C_2^k k! N^{2k}. }
Consequently, there exist $c > 0$ and $C$ such that
for all $\lambda > 0$ we have
\be \label{e:MNexp}
{ \bP \big( M_N \ge \lambda N^2 \big)
  \le C \exp ( - c \lambda ). }
\ee
\end{lemma}

\begin{remark}
If $M^S_N$ is the length of a simple random walk path
run until its first exit from $Q_N$ then it is well known that 
$M^S_N/N^2$ has an exponential tail. However 
we do not have $M_N \le M^S_N$, so need an alternative argument
to obtain the bound \eqref{e:MNexp}.
\end{remark}

\begin{proof}[Proof of Lemma \ref{L:len-box-ub}.]
We have
\eqnsplst
{ \bE \big( M_N^k \big)
  &\le \bE \big( |S[0,\infty) \cap Q_N|^k \big) \\
  &= k! \sum_{x_1, \dots, x_k \in Q_N} 
    G(0,x_1) G(x_1,x_2) \dots G(x_{k-1},x_k) \\
  &\le k! \Big( \sum_{z \in Q_{2N}} G(0,z) \Big)^k \\
  &= C_2^k k! N^{2k}. }
To see the second statement:
\eqnst
{ \bP \big( M_N \ge \lambda N^2 \big)
  \le \exp (- \lambda t N^2) \bE \big( e^{t M_N} \big)
  \le \exp ( - \lambda t N^2 ) \frac{1}{1 - C_2 t N^2}. }
Choosing $t = 1/(2 C_2 N^2)$ completes the proof of the Lemma.
\end{proof}

\begin{proof}[Proof of Theorem \ref{T:UST-lb}]
It is sufficient to prove the statement for 
$0 < \lambda < \lambda_0$ for some fixed $\lambda_0$.
Let us choose $N = \lambda^{\alpha} \sqrt{n}$ with some
exponent $\alpha > 0$, that we will optimize over at the
end of the proof. We have
\eqnst
{ \bP ( M_N \ge n/2 )
  \le C \exp \Big( - c \frac{n}{2 N^2} \Big)
  = C \exp ( - c \lambda^{-2 \alpha} ). }
Condition on $L$, as in the proof of Theorem \ref{T:cyclepop}, 
and assume the event
\eqnst
{ \widetilde{G}
  = G(c_1,c_2,C_1) \cap \{ M_N < n/2 \}. } 
We set
\eqnst
{ \lambda n^2 
  = c_3 k m^4
  \asymp N m^3, }
which means we pick $m$ to be
\eqnst
{ m
  \asymp \sqrt{n} \lambda^{(1-\alpha)/3}. }
Hence $N / m \asymp k \asymp \lambda^{(4 \alpha - 1)/3}$.
Note that this implies that 
\eqnst
{ \bP \big( G(c_1,c_2,C_1)^c \big)
  \le C \exp ( -c (N/m) )
  = C \exp ( -c \lambda^{(4 \alpha - 1)/3} ). }
Since we want $N / m \gg 1$, we impose the
condition $0 < \alpha < 1/4$ on $\alpha$.

For each $j \in J$, let 
\eqnsplst
{ \widetilde{V}^I_j 
  &= \left\{ y \in A'(j) : 
    \parbox{8cm}{cycle popping using $\mathsf{Stacks\ I}$
    reveals a path from $y$ to $L$ of lenght $\le n/2$} \right\} \\
  \widetilde{V}^{II}_j
  &= \left\{ y \in A'(j) : 
    \parbox{8cm}{cycle popping using $\mathsf{Stacks\ II}$
    reveals a path from $y$ to $A''(j) \cap \beta_j$
    of length $\le n/2$} \right\} }
Notice that $(\widetilde{V}^I_j,\widetilde{V}^{II}_j)_{j \in J}$
are again conditionally independent, given $L$.
The same proof as in Lemma \ref{L:coupling} shows that 
we have $\widetilde{V}^I_j \supset \widetilde{V}^{II}_j$ 
for all $j \in J$. 

In estimating $\bE \big( \widetilde{V}^{II} \big)$ from below, 
we write
\eqnspl{e:first-moment-lb}
{ \bP \big( y \in \widetilde{V}^{II}_j \,\big|\, L,\, j \in J \big) 
  &\ge \bP \big( T_{A''(j) \cap \beta_j}[S^y] < \tau_{A''(j)}[S^y] 
     \,\big|\, L,\, j \in J \big) \\
  &\qquad - \bP \big( | \sE^F_{\pd A''(j)}(S^y) | > n/2,\, T_{A''(j) \cap \beta_j}[S^y] \circ \Theta_{n/2} < \infty \big). }
The first term on the right hand side is $\ge c m^{4-d}$ due to \eqref{e:hit-first} and $j \in J$.
We now show that the subtracted term is $\le C \exp ( - c n/m^2 ) m^{4-d}$.

Note that we may restrict to $n/2 > 2 m^2$ for convenience (although not needed for the claim), 
since our choice of $m$ implies that $n \asymp m^2 \lambda^{-2(1 - \alpha)/3}$, and we 
are considering small $\lambda$.
Using the Markov property at time $n/2 - m^2$, the second term 
in the right hand side of \eqref{e:first-moment-lb} is at most
\eqnst
{ \bP^y \big( \tau_{A''(j)} > n/2 - m^2 \big) \,
     \sum_{z \in A''(j)} \bP^z \big( T_{A''(j) \cap \beta_j} < \infty \big) \,
     \bP^y \big( S(n/2) = z \,\big|\, \tau_{A''(j)} > n/2 - m^2 ). }
The first probability can be bounded by $C \exp ( - c n/m^2 )$, by considering 
stretches of the walk of length $m^2$, in each of which there is probability 
$\ge c > 0$ of exit from $A''(j)$. The conditional distribution of $z$ is 
bounded above by $c m^{-d}$, due to the local CLT applied to 
$S(n/2-m^2), \dots, S(n/2)$. Hence we are left to show that 
\eqnst
{ \sum_{z \in A''(j)} \bP^z \big( T_{A''(j) \cap \beta_j} < \infty \big)
  \le m^4. }

Let us write $\wt \beta_j = A''(j) \cap \beta_j$, and 
$h(z) = \bP^z ( T_{\wt \beta_j} < \infty )$. By a last exit decomposition
$h(z) = \sum_{u \in \wt \beta_j} G(z,u) e_{\wt \beta_j}(u)$, where
$e_{\wt \beta_j}(u) = \bP^u ( T^+_{\wt \beta_j} = \infty )$. Therefore,
we have
\eqnsplst
{ \sum_{z \in A''(j)} h(z)
  &= |\wt \beta_j| + \sum_{z \in A''(j) \setminus \wt \beta_j} h(z)
  \le C m^2 + \sum_{u \in \wt \beta_j} \sum_{z \in A''(j)} G(z,u) e_{\wt \beta_j}(u) \\
  &\le C m^2 + C m^2 \sum_{u \in \wt \beta_j} e_{\wt \beta_j}(u)
  = C m^2 + C m^2 \Cap(\wt \beta_j)
  \le C m^4, }
using that $|\wt \beta_j|, \Cap(\wt \beta_j) \le C m^2$ when $j \in J$.

Hence we obtain that there exists $\lambda_0 = \lambda_0(d) > 0$, such that 
when $0 < \lambda \le \lambda_0$, the right hand side of \eqref{e:first-moment-lb} 
is at least
\eqnsplst
{ c m^{4-d} - C \, \exp ( -c n/m^2 ) \, m^{4-d}
  \ge c m^{4-d} - C \, \exp ( -c \lambda^{- 2 (1 - \alpha)/3} ) \, m^{4-d} 
  \ge c m^{4-d}. }
It follows that $\bE \big( |\wt V^{II}_j| \,\big|\, L,\, j \in J \big) \ge c m^4$.

For the second moment, we simply estimate
\eqnst
{ \bE \big( (\widetilde{V}^{II}_j)^2 \,\big|\, L,\, j \in J \big)
  \le \bE \big( (V^{II}_j)^2 \,\big|\, L,\, j \in J \big)
  \le C m^8. }
The one-sided Chebyshev inequality yields that for
some $c_4 = c_4 (d) > 0$ we have
\eqnst
{ \bP \big( \widetilde{V}^{II}_j \ge c_4 m^4 \,\big|\, L,\, j \in J \big)
  \ge c 
  > 0. }
This allows us to complete the proof as follows.
\eqnsplst
{ &\bP \big( |B_\sU(0,n)| \le \lambda n^2 \big) \\
  &\qquad \le \bP \big( \widetilde{G}^c \big) 
      + \bP \Big( \widetilde{G},\, 
      \sum_{j \in J} \widetilde{V}^{I}_j \le \lambda n^2 \Big) \\
  &\qquad \le \bP ( M_N > n/2 ) + \bP \big( G(c_1,c_2,C_1)^c \big) 
      + \bE \Big( \bP \Big( \sum_{j \in J} \widetilde{V}^{II}_j < c_3 k m^4 
          \,\Big|\, L \Big) ; \widetilde{G} \Big) \\
  &\qquad \le C \exp ( -c \lambda^{-2 \alpha} )
      + C \exp ( -c \lambda^{(4 \alpha - 1)/3} )
      + \exp ( -c \lambda^{(4 \alpha - 1)/3} ). }
We choose $\alpha$, so that $- 2 \alpha = (4 \alpha - 1)/3$,
so $\alpha = 1/10$. This completes the proof of the Theorem.
\end{proof}

\begin{remark}
We note the following minor extension of Theorem \ref{T:cyclepop}, that is 
needed in \cite{BHJ15}. Similarly to Remark \ref{R:cond-exit}, since the 
arguments of Theorem \ref{T:cyclepop} only rely on properties of 
$\sE^F_{Q_N^c} L$, the result extends to the case when the 
component of the origin is connected to a fixed vertex $w \in \partial D$.
\end{remark}

\noindent MB: Department of Mathematics,
University of British Columbia,
Vancouver, BC V6T 1Z2, Canada. \\
{\tt barlow@math.ubc.ca}

\ \\

\noindent AAJ: Department of Mathematical Sciences,
University of Bath,
Claverton Down, Bath BA2 7AY,
United Kingdom. \\
{\tt A.Jarai@bath.ac.uk}

\end{document}